\documentclass[10pt,a4paper]{amsart}
\usepackage{amsmath}
\usepackage{amsfonts}
\usepackage{amssymb}
\usepackage{graphicx}
\usepackage[utf8]{inputenc}
\usepackage{pdflscape}
\usepackage{float}
\usepackage{mathtools}
\usepackage{appendix}
\usepackage[colorlinks=true, linkcolor=red, citecolor=blue]{hyperref}
\usepackage{epsfig}
\usepackage{tikz}
\usepackage{mathrsfs}

\numberwithin{equation}{section}
\makeatletter
\@namedef{subjclassname@2010}{\textup{2020} Mathematics Subject Classification}
\makeatother

\newtheorem{theorem}{Theorem}[section]
\newtheorem{definition}{Definition}[section]
\newtheorem{lemma}{Lemma}[section]
\newtheorem{proposition}{Proposition}[section]
\newtheorem{corollary}{Corollary}[section]
\newtheorem{remark}{Remark}[section]
\newtheorem{claim}{\textit{Claim}}[section]

\begin{document}
\title[KdV equation: A study on star-graphs]{An analysis of the entire functions associated with the operator of the KdV equation}
	\author[Capistrano--Filho]{Roberto de Almeida Capistrano--Filho}
	\author[Parada]{Hugo Parada}
	\author[da Silva]{Jandeilson Santos da Silva$^*$}

	\address{Departamento de Matem\'atica, Universidade Federal de Pernambuco, S/N Cidade Universit\'aria, 50740-545, Recife (PE), Brazil}
	\email{roberto.capistranofilho@ufpe.br}
	\email{jandeilson.santos@ufpe.br}
	
	\address{Universit\'e Paul Sabatier, Institut de Math\'ematiques de Toulouse, 118 Route de Narbonne, 31062 Toulouse, France}
	\email{hugo.parada@math.univ-toulouse.fr}
	
	\thanks{$^*$Corresponding author: {jandeilson.santos@ufpe.br}}
	\thanks{Capistrano–Filho was partially supported by CNPq grant numbers 301744/2025-4, 421573/2023-6, CAPES/COFECUB grant number 88887.879175/2023-00, and PROPG (UFPE) \textit{via} PROAP resources. Da Silva acknowledges support from CNPq. Parada was funded by the Labex CIMI postdoctoral fellowship under the grant agreement ANR-11-LABX-0040.}	
	
	\begin{abstract}		
	It is well known that the controllability property of partial differential equations (PDEs) is closely linked to the proof of an observability inequality for the adjoint system, which, sometimes, involves analyzing a spectral problem associated with the PDE under consideration. In this work, we study a series of spectral issues that ensure the controllability of the renowned Korteweg-de Vries equation on a star-graph. This investigation reduces to determining when certain functions, associated with this spectral problem, are entire. The novelty here lies in presenting this detailed analysis in the context of a star graph structure.
	\end{abstract}
	
	\date{\today}
	\maketitle
	
	
	\thispagestyle{empty}
	
	\section{Introduction}

In mathematics and physics, a quantum graph is a network structure consisting of vertices connected by edges, where each edge hosts a differential or pseudo-differential equation. When each edge is equipped with a natural metric, the graph is called a metric graph. A typical example is a power network, where power lines (edges) connect transformer stations (vertices); in this case, the differential equations could represent the voltage along each line, with boundary conditions at the vertices ensuring that the sum of currents across all edges at each vertex is zero.

Linus Pauling initially studied quantum graphs in the 1930s as models of free electrons in organic molecules. They also appear in various mathematical contexts, such as quantum chaos, waveguide studies, photonic crystals, and Anderson localization—the phenomenon where wave diffusion is absent in a disordered medium. Quantum graphs are also used as a limiting model for shrinking thin wires. They have become significant in mesoscopic physics as theoretical tools for understanding nanotechnology. Another simplified version of quantum graphs was introduced by Freedman \textit{et al.} in \cite{Freedman}.

Beyond solving the differential equations for specific applications, typical questions in quantum graph theory include well-posedness, controllability, and identifiability. For instance, controllability might involve determining the necessary inputs to bring a system to a desired state, such as ensuring sufficient power supply throughout a network. Identifiability could involve selecting measurement points, like pressure gauges in a water pipe network, to detect issues such as leaks.

\subsection{Dispersive systems on graph structure}
In recent years, the study of nonlinear dispersive models on metric graphs has gained significant attention from mathematicians, physicists, chemists, and engineers (see \cite{BK, BlaExn08, BurCas01, K, Mug15} for details and further references). A primary framework for modeling these phenomena is the star graph $\mathcal{G}$, a metric graph with \( N \) half-lines of the form \( (0, +\infty) \) that connect at a common vertex \( \nu=0 \). On each edge, a nonlinear equation, such as the nonlinear Schrödinger equation, is defined (see works by Adami \textit{et al.} \cite{AdaNoj14, AdaNoj15} and Angulo and Goloshchapova \cite{AngGol17a, AngGol17b}). Introducing nonlinearities in dispersive models of such networks creates a rich field for exploring soliton propagation and nonlinear dynamics. A key challenge in this analysis is the vertex, where the star graph may exhibit bifurcation or multi-bifurcation behavior, especially in more complex graph structures.

Other nonlinear dispersive systems on graphs also yield interesting results. For instance, regarding the well-posedness theory, Cavalvante \cite{Cav} studied the local well-posedness for the Cauchy problem of the Korteweg-de Vries equation on a metric star graph with a negative half-line and two positive half-lines meeting at a common vertex \( \nu=0 \) (the so-called \( \mathcal{Y} \)-junction). Another example is the Benjamin–Bona–Mahony (BBM) equation. Bona and Cascavel \cite{bona} established local well-posedness in the Sobolev space \( H^1 \), while Mugnolo and Rault \cite{Mugnolobbm} demonstrated the existence of traveling waves for the BBM equation on graphs. Using an alternative approach, Ammari and Crépeau \cite{AmCr1} derived results for well-posedness and stabilization of the BBM equation in a star-shaped network with bounded edges. It is important to point out that Mugnolo et al. \cite{MugNoSe} obtained a characterization of all boundary conditions under which the Airy-type evolution equation $u_t=\alpha u_{x x x}+\beta u_x$, for $\alpha \in \mathbb{R} \backslash\{0\}$ and $\beta \in \mathbb{R}$ on star graphs, generates contraction semigroups.

Notably, notable contributions have been made in the areas of control theory and inverse problems. Ignat \textit{et al.} \cite{Ignat2011} investigated the inverse problem for the heat and Schrödinger equations on a tree structure. Later, Baudouin and Yamamoto \cite{Baudouin} introduced a unified and simplified approach to the inverse problem of coefficient determination. Additionally, Ammari and Crépeau \cite{Ammari and Crepeau 2018}, Cerpa \textit{et al.} \cite{Cerpa, Cerpa1}, and Parada \textit{et al.} \cite{Parada2022a,Parada2022b} proved results on stabilization and boundary controllability for the KdV equation on star-shaped graphs.

\subsection{Background of control theory} Research on the control and stabilization of the KdV equation originated with the work of Russell and Zhang \cite{russell3, russell2, Russell1, zhang4}, who examined internal control of the KdV equation under periodic boundary conditions. Since then, extensive studies have focused on the control and stabilization of the KdV equation (see \cite{cerpatut, coron, GG1, Rosier, RZsurvey, zhang2} and references therein). 

Considering a bounded domain, the control problem was presented in a pioneering work of Rosier \cite{Rosier} that studied the following system 
\begin{equation}
\begin{cases}
\partial_{t}u+\partial_{x}u+u\partial_{x}u+\partial_{x}^{3}u=0& \text{ in } (0,L)\times(0,T),\\
u(0,t)=0,\text{ }u(L,t)=0,\text{ }\partial_{x}u(L,t)=g(t) &  \text{ in }(0,T),\\
u(x,0)=u_0(x) & \text{ in }(0,L),
\end{cases} \label{2}
\end{equation}
where the boundary value function $g(t)$ is considered as a control input. The author answered the following problem for the system \eqref{2}.

\vspace{0.2cm}
\noindent\textbf{Exact controllability problem:} \textit{Given $T>0$ and $u_0,u_T\in L^2(0,L)$, can one find  an appropriate control input $g(t)\in L^2(0,T)$ such that the corresponding solution $u(x,t)$ of \eqref{2} satisfies
\begin{equation}\label{ControlCon}
u(x,0)=u_0(x) \quad \text{and} \quad u(x,T)=u_T(x)?
\end{equation}}

Rosier \cite{Rosier} investigated boundary control of the KdV equation on the finite domain \((0, L)\) with Dirichlet boundary conditions \eqref{2}. He first analyzed the associated linear system:
\begin{equation}\label{2a}
\begin{cases}
\partial_{t}u + \partial_{x}u + \partial_{x}^{3}u = 0, & \text{in } (0, L) \times (0, T), \\
u(0, t)=0,\text{ } u(L,t)=0, \; \partial_{x}u(L, t) = g(t), & \text{in } (0, T), \\
u(x, 0) = u_0(x), & \text{in } (0, L),
\end{cases}
\end{equation}
and discovered the \textit{critical length phenomenon}, where the exact controllability of system \eqref{2a} depends on the domain length \( L \). Specifically, exact controllability in \(L^2(0, L)\) is achieved if and only if \( L \) does not belong to the set
\begin{equation}\label{critical}
\mathcal{N} := \left\{ \frac{2\pi}{\sqrt{3}} \sqrt{k^2 + kl + l^2} \, : \, k, l \in \mathbb{N}^* \right\}.
\end{equation}
Using a fixed-point argument, Rosier extended this controllability result from the linear to the nonlinear system for cases where \( L \notin \mathcal{N} \).

In the case where \( L \in \mathcal{N} \), Rosier demonstrated in \cite{Rosier} that the associated linear system \eqref{2a} is not controllable. Specifically, there exists a finite-dimensional subspace of \( L^2(0, L) \), denoted as \( \mathcal{M} = \mathcal{M}(L) \), that cannot be reached from \( 0 \) by the linear system. More precisely, for any nonzero state \( \psi \in \mathcal{M} \), if \( g \in L^2(0, T) \) and \( u \in C([0, T]; L^2(0, L)) \cap L^2(0, T; H^1(0, L)) \) satisfies \eqref{2a} with the initial condition \( u(\cdot, 0) = 0 \), then \( u(\cdot, T) \neq \psi \). A spatial domain \((0, L)\) is termed \textit{critical} for system \eqref{2a} if its length \( L \) belongs to \( \mathcal{N} \).

\vspace{0.2cm}

When the spatial domain $(0,L)$ is critical, one usually would not expect the corresponding nonlinear system \eqref{2} to be exactly controllable, as the linear system \eqref{2a} is not. It thus came as a surprise when Coron and Cr\'epeau showed  in \cite{coron} that the nonlinear system (\ref{2}) is  still locally exactly controllable  even though its spatial domain is critical with its length   $L=2k\pi$ and  $k\in\mathbb{N}^{*}$ satisfying \[\not \exists(m,n)\in\mathbb{N}^{*}\times\mathbb{N}^{*} \text{}\text{ with }\text{} m^2+mn+n^2=3k^2 \text{}\text{ and }\text{} m\neq n.\]
For those values of $L$, the unreachable space $\mathcal{M}$ of the associated linear system is a one-dimensional linear space generated by the function $1-\cos(x)$.  As for the other types of critical domains,   the nonlinear system (\ref{2})  was shown later by Cerpa \cite{cerpa}, and  Cerpa and Cr\'epeau in \cite{cerpa1} to be local, large time exactly controllable.

It is important to point out that if we change the control of position in the boundary condition of \eqref{2a}, for example 
\begin{equation*}
u(0,t)=h(t),\text{ }u(L,t)=0,\text{ }\partial_{x}u(L,t)=0 \quad \text{ in }(0,T)
\end{equation*}
or 
\begin{equation}\label{2c}
u(0,t)=0,\text{ }u(L,t)=f(t),\text{ }\partial_{x}u(L,t)=0\quad \text{ in }(0,T),
\end{equation}
then we can not explicitly characterize the critical sets for the KdV equation. However, for the boundary conditions in \eqref{2c}, it was established in \cite{GG1} the existence of a countable set of critical lengths $\mathcal{N}^*\subset (0,+\infty)$ which can be described as follows
	\begin{align}\label{critical-2}
		\mathcal{N}^*=\left\{L \in \mathbb{R}^+\backslash\ \{0\};
		\begin{array}{l}
			\exists(a,b)\in \mathbb{C}^2\text{ \ such that \ }ae^a=be^b=-(a+b)e^{-(a+b)}\\
			\text{ and }L^2=-\left(a^2+ab+b^2\right)
		\end{array}
		\right\}
	\end{align}
	For more details, we mention \cite{cerpatut,GG1}.

After $97'$, some authors tried to prove the critical set phenomenon for the KdV equation with some boundary condition, and we can cite, for example, \cite{GG1,CeRiZh}, and the references therein. However, for this set considered in these works, the authors were not allowed to explicitly characterize the set. Twenty years later, in \cite{Capistrano 2017}, another boundary condition was considered. The authors introduced the KdV equation with Neumann conditions. Capistrano-Filho \textit{et al.} investigated the following boundary control system
\begin{equation}
\begin{cases}
\partial_{t}u+\partial_{x}u+u\partial_{x}u+\partial_{x}^{3}u=0&  \text{ in } (0,L)\times(0,T),\\
\partial_{x}^{2}u(0,t)=0,\text{ }\partial_{x}^{2}u(L,t)=0,\text{ }\partial_{x}u(L,t)=h(t)& \text{ in }(0,T),\\
u(x,0)=u_0(x)& \text{ in }(0,L).
\end{cases}  \label{1}
\end{equation}
First, the authors studied the following linearized system associated with (\ref{1}),
\begin{equation}
\begin{cases}%
\partial_{t}u+ (1+\beta )\partial_{x}u+\partial_{x}^{3}u=0 & \text{ in } (0,L)\times(0,T),\\
\partial_{x}^{2}u(0,t)=0,\text{ }\partial_{x}^{2}u(L,t)=0,\text{ }\partial_{x}u(L,t)=h(t)& \text{ in }(0,T),\\
u(x,0)=u_0(x)& \text{ in }(0,L),
\end{cases} \label{linear}
\end{equation}
where $\beta $ is a given real constant. For any $\beta \ne -1$,  considering the following set
\begin{equation*}
\mathcal{R}_{\beta} :=\left\{  \frac{2\pi}{\sqrt{3(1+\beta)}}\sqrt{k^{2}+kl+l^{2}}
\,:k,\,l\,\in\mathbb{N}^{\ast}\right\}\cup\left\{\frac{k\pi}{\sqrt{\beta +1}}:k\in\mathbb{N}^{\ast}\right\}.
\end{equation*}
The authors showed the following results:
\begin{itemize}
\item[(i)] If $\beta \ne -1$, the linear system (\ref{linear}) is exactly controllable in the space $L^2 (0,L)$  if and only if the length L of the spatial domain $(0, L)$ does not belong to the set $\mathcal{R}_{\beta}$.
\item[(ii)] If $\beta =-1$, then the system (\ref{linear}) is not exactly controllable in the space $L^2 (0,L)$ for any $L>0$.
\end{itemize}

As done in \cite{Rosier}, the controllability of the nonlinear system \eqref{1} holds using a fixed point argument.  Moreover, the set $\mathcal{R}_{\beta}$ is completely characterized. Note that, when $\beta=0$,  $\mathcal{N} $ (see \eqref{critical}) is a proper subset of $\mathcal {R}_{0}$.  The linear system (\ref{linear}) has more critical length domains than that of the linear system (\ref{2a}). In the case of $\beta =-1$, every $L>0$ is critical for the system (\ref{linear}). By contrast, removing the term $\partial_{x}u$ from the equation in (\ref{2a}),  every $L>0$ is not critical for the system (\ref{2a}).
\\

To conclude this section, we also define the following set
\begin{align}\label{critical-3}
		\mathcal{N}^\dagger=\left\{L \in \mathbb{R}^+\backslash\ \{0\};
		\begin{array}{l}
			\exists(a,b)\in \mathbb{C}^2\text{ \ such that \ }a^{2}e^a=b^{2}e^b=(a+b)^{2}e^{-(a+b)}\\
			\text{ and }L^2=-\left(a^2+ab+b^2\right)
		\end{array}
		\right\}
	\end{align}
As we will prove in the Appendix~\ref{app: countable ciritcal}, the set $\mathcal{N}^\dagger$ is a countable, non-empty new critical set in the literature of controllability of KdV equations.
\subsection{Setting of problem and functional framework}
This work focuses on investigating the exact controllability of the linear KdV equation on a star-shaped network. \noindent The equation governing the dynamics is given by
\begin{equation*}
\begin{cases}
\partial_t u_j(t, x) + \partial_x u_j(t, x) + \partial_x^3 u_j(t, x) = 0, & t \in (0, T), \ x \in (0, l_j), \ j=1,\dots,N,\\
u_j(t, 0) = u_1(t, 0), & t \in (0, T), \ \forall j=2,\dots,N, \\
\displaystyle\sum_{j=1}^N \partial_x^2 u_j(t, 0) = -\alpha u_1(t, 0) + g_0(t), & t \in (0, T), \\
u_j(t, l_j) = 0, \quad \partial_x u_j(t, l_j) = g_j(t), & t \in (0, T), \ j=1,\dots,N, \\
u_j(0, x) = u_j^0(x), & x \in (0, l_j),
\end{cases}
\end{equation*}
where \(\alpha > \frac{N}{2}\). The conditions at the central node are motivated by previous studies \cite{Ammari and Crepeau 2018,Cav,Parada2022a,Parada2022b}. Following \cite{Parada2022b}, \(u_j\) represents the dimensionless, scaled deflection from the rest position, and \(v_j\) denotes the velocity along a branch \(j\) of long water waves. Thus, we have:
\begin{equation}\label{inta}
\begin{cases}
\partial_t u_j + \partial_x u_j + \partial_x^3 u_j + u_j \partial_x u_j = 0, & x \in (0, l_j), \ t \in (0, T), \ j=1,\dots,N, \\
v_j = u_j - \frac{1}{6} u_j^2 + 2 \partial_x^2 u_j, & x \in (0, l_j), \ t \in (0, T), \ j=1,\dots,N.
\end{cases}
\end{equation}
Assuming the water level at the central node is constant, and the net flux is zero, the following natural conditions arise
$$u_j(t, 0) = u_1(t, 0), \quad t \in (0, T), \ j=2,\dots,N,$$
and
$$\sum_{j=1}^N u_j(t, 0) v_j(t, 0) = 0, \quad t \in (0, T).$$
Linearizing \eqref{inta} around zero yields the following boundary conditions
$$u_j(t, 0) = u_1(t, 0), \quad t \in (0, T), \ j=2,\dots,N,$$
and
$$
\sum_{j=1}^N \partial_x^2 u_j(t, 0) = -\frac{N}{2} u_1(t, 0), \quad t \in (0, T).
$$
	
With this context, given $N\geq 2$, consider a star-shaped network $\mathcal{T}$ with $N$ edges described by intervals $I_j=(0,l_j)$, $l_j>0$ for $j=1,\dots,N$. Denoting by $e_1,...,e_N$ the edges of $\mathcal{T}$ one have $\mathcal{T}=\bigcup_{j=1}^Ne_j$.  To be precise, in this work, we will study the control properties for $N$ linear KdV equations posed on $\mathcal{T}$ with the following boundary conditions
\begin{align}\label{LkdV}
\begin{cases}
\partial_{t}u_{j}+\partial_{x}u_{j}+\partial_{x}^{3}u_{j}=0,&t \in (0,T),\ x \in (0,l_j),\ j=1,\dots,N,\\
u_j(t,0)=u_1(t,0),&t\in (0,T),\ j=2,\dots,N,\\
\displaystyle\sum_{j=1}^N\partial_{x}^{2}u_{j}(t,0)=-\alpha u_1(t,0),&t \in (0,T)\\	
u_j(t,l_j)=p_j(t),\ \ \ \ \ \partial_{x}u_{j}(t,l_j)=g_j(t),&t\in (0,T),\ j=1,\dots,N,\\
u_j(0,x)=u_j^0(x),&x \in (0,l_j),\ j=1,\dots,N,
\end{cases}
\end{align}
where $u=(u_1,...,u_N)$ stands the state of the system, $p_j$ and $g_j$ are the controls inputs and $\alpha>\frac{N}{2}$. This system was studied in \cite{Parada 2023}, where the null controllability was obtained using $2N-2$ controls.  

Our goal in this article is to prove the controllability of \eqref{LkdV} using a reduced number of boundary controls. In most of our results, we assume $l_j=L>0$ for all $j=1,\dots,N$ and we are interested in answering the following question:
	
	\vspace{0.2cm}
\noindent\textbf{New boundary conditions:} \textit{Are there other boundary conditions such that the controllability for the system \eqref{LkdV} holds?}
\vspace{0.2cm}

We will present here six new possibilities that ensure boundary controllability properties for the system \eqref{LkdV}, and that exhibit a relation of the length $l_j$ with a critical set, that is, in some appropriate set of boundary conditions, the controllability holds if and only if the length $l_j$ avoids specific values or for all $l_j>0$.  Given $0\leq m\leq N$, we put Neumann controls on the first $m$ edges and Dirichlet controls on the remaining $N-m$ edges. With this in mind, we will analyze the problem under the assumption $\alpha=N$ and $l_{j}=L$ for $j=1,\dots,N$ in six situations (with different results). The first one considers $N=2$ and $m=1$ as it is illustrated in Figure~\ref{fig:N=2}. We prove that the corresponding system \eqref{LkdV} is exactly controllable for any $L>0$. 		
	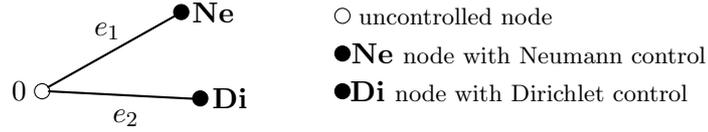
\begin{figure}[H]
		\centering
		\begin{tikzpicture}
			\draw(0, 0) circle (1mm);
			\draw[line width=0.8pt](0.07, 0.05) -- (1.75, 1);
			\draw[fill=black](1.83, 1.05) circle (1mm);
			\draw[line width=0.8pt](0.08, 0) -- (2, -0.1);
			\draw[fill=black](2.08, -0.095) circle (1mm);
			\draw node (no1) at (-0.3,0) {$0$};
			\draw node (no1) at (0.86, 0.8) {$e_1$};
			\draw node (no1) at (1.1, -0.35) {$e_2$};
			\draw node (no1) at (2.25, 1.05) {\textbf{Ne}};
			\draw node (no1) at (2.47, -0.095){\textbf{Di}};
			\draw(3.95,1) circle (1mm);
			\draw node (no1) at (6.4,1) {{\footnotesize uncontrolled node}\ \ \ \ \ \ \ \ \ \ \ \ \ \ \ \ };
			\draw[fill=black](3.95,0.5) circle (1mm);
			\draw node (no1) at (6.4,0.5) {\textbf{Ne} {\footnotesize node with Neumann control}};
			\draw[fill=black](3.95,0) circle (1mm);
			\draw node (no1) at (6.4,0) {\textbf{Di} {\footnotesize node with Dirichlet control}\ \ \ };
		\end{tikzpicture}
		\caption{Network with $2$ edges and mixed controls.}
        \label{fig:N=2}
	\end{figure}
The second situation considers $N\geq 3$, $m=1$, and it is illustrated in Figure~\ref{fig:m=1}. In this case, we can prove the exact controllability to the corresponding system \eqref{LkdV} when $L\notin\mathcal{N}^*$, where $\mathcal{N}^*$ is defined by \eqref{critical-2}.
\begin{figure}[H]
	\centering
	\begin{tikzpicture}
		\draw(0, 0) circle (1mm);
		\draw[line width=0.8pt](0.02, 0.09) -- (0.5, 1.95);
		\draw[fill=black](0.53, 2.04) circle (1mm);
		\draw[line width=0.8pt](0.07, 0.07) -- (1.33, 1.5);
		\draw[fill=black](1.39, 1.57) circle (1mm);
		\draw[line width=0.8pt](0.1, 0.04) -- (1.85, 0.78);
		\draw[fill=black](1.93, 0.815) circle (1mm);
		\draw[line width=0.8pt](0.08, -0.04) -- (1.75, -1);
		\draw[fill=black](1.825, -1.045) circle (1mm);
		\draw[line width=0.8pt](0.05, -0.09) -- (1.02, -1.73);
		\draw[fill=black](1.07, -1.8) circle (1mm);
		\draw[fill=black](1.45, 0.2) circle (0.3mm);
		\draw[fill=black](1.45, 0.05) circle (0.3mm);
		\draw[fill=black](1.45, -0.1) circle (0.3mm);
		\draw node (no1) at (-0.3,0) {$0$};
		\draw node (no1) at (0.19, 1.7) {$e_1$};
		\draw node (no1) at (0.86, 1.35) {$e_2$};
		\draw node (no1) at (1.35, 0.8) {$e_3$};
		\draw node (no1) at (1.7, -0.6) {$e_{_{N-1}}$};
		\draw node (no1) at (1.06, -1.25) {$e_{_N}$};
		\draw node (no1) at (0.95, 2.05) {\textbf{Ne}};
		\draw node (no1) at (1.79, 1.57) {\textbf{Di}};
		\draw node (no1) at (2.3, 0.815) {\textbf{Di}};
		\draw node (no1) at (2.2, -1.045) {\textbf{Di}};
		\draw node (no1) at (1.46, -1.8) {\textbf{Di}};
		\draw(3.95,-0.5) circle (1mm);
		\draw node (no1) at (6.4,-0.5) {{\footnotesize uncontrolled node}\ \ \ \ \ \ \ \ \ \ \ \ \ \ \ \ };
		\draw[fill=black](3.95,-1) circle (1mm);
		\draw node (no1) at (6.4,-1) {\textbf{Ne} {\footnotesize node with Neumann control}};
		\draw[fill=black](3.95,-1.5) circle (1mm);
		\draw node (no1) at (6.4,-1.5) {\textbf{Di} {\footnotesize node with Dirichlet control}\ \ \ };
	\end{tikzpicture}
	\caption{Network with $N$ edges for $m=1$.}	
    \label{fig:m=1}
\end{figure}
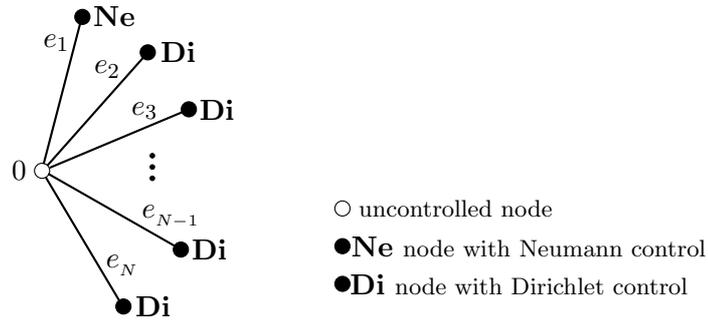
In the third case, we consider $N\geq 3$ and $m=N-1$, as illustrated in Figure~\ref{fig:m=N-1}. For this situation, the exact controllability to the corresponding system \eqref{LkdV} holds if and only if $L\notin\mathcal{N}$, where $\mathcal{N}$ is defined by \eqref{critical}. 
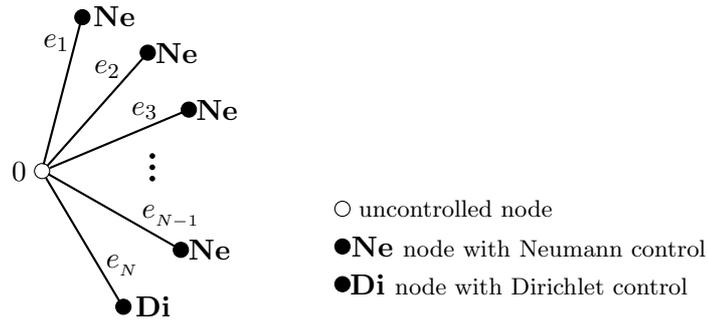
\begin{figure}[H]
		\centering
		\begin{tikzpicture}
			\draw(0, 0) circle (1mm);
			\draw[line width=0.8pt](0.02, 0.09) -- (0.5, 1.95);
			\draw[fill=black](0.53, 2.04) circle (1mm);
			\draw[line width=0.8pt](0.07, 0.07) -- (1.33, 1.5);
			\draw[fill=black](1.39, 1.57) circle (1mm);
			\draw[line width=0.8pt](0.1, 0.04) -- (1.85, 0.78);
			\draw[fill=black](1.93, 0.815) circle (1mm);
			\draw[line width=0.8pt](0.08, -0.04) -- (1.75, -1);
			\draw[fill=black](1.825, -1.045) circle (1mm);
			\draw[line width=0.8pt](0.05, -0.09) -- (1.02, -1.73);
			\draw[fill=black](1.07, -1.8) circle (1mm);
			\draw[fill=black](1.45, 0.2) circle (0.3mm);
			\draw[fill=black](1.45, 0.05) circle (0.3mm);
			\draw[fill=black](1.45, -0.1) circle (0.3mm);
			\draw node (no1) at (-0.3,0) {$0$};
			\draw node (no1) at (0.19, 1.7) {$e_1$};
			\draw node (no1) at (0.86, 1.35) {$e_2$};
			\draw node (no1) at (1.35, 0.8) {$e_3$};
			\draw node (no1) at (1.7, -0.6) {$e_{_{N-1}}$};
			\draw node (no1) at (1.06, -1.25) {$e_{_N}$};
			\draw node (no1) at (0.95, 2.05) {\textbf{Ne}};
			\draw node (no1) at (1.77, 1.57) {\textbf{Ne}};
			\draw node (no1) at (2.3, 0.815) {\textbf{Ne}};
			\draw node (no1) at (2.2, -1.045) {\textbf{Ne}};
			\draw node (no1) at (1.46, -1.8) {\textbf{Di}};
			\draw(3.95,-0.5) circle (1mm);
			\draw node (no1) at (6.4,-0.5) {{\footnotesize uncontrolled node}\ \ \ \ \ \ \ \ \ \ \ \ \ \ \ \ };
			\draw[fill=black](3.95,-1) circle (1mm);
			\draw node (no1) at (6.4,-1) {\textbf{Ne} {\footnotesize node with Neumann control}};
			\draw[fill=black](3.95,-1.5) circle (1mm);
			\draw node (no1) at (6.4,-1.5) {\textbf{Di} {\footnotesize node with Dirichlet control}\ \ \ };
		\end{tikzpicture}
		\caption{Network with $N$ edges for $m=N-1$.}	
        \label{fig:m=N-1}
	\end{figure}
In the fourth case, we study the problem under the configuration $N>3$ and $1<m<N-1$. Figure~\ref{fig:m=2} illustrates this situation for $m=2$. In this case, we can prove the exact controllability to the corresponding system \eqref{LkdV} when $L\notin\mathcal{N}\cup \mathcal{N}^*$.
\begin{figure}[H]
	\centering
	\begin{tikzpicture}
		\draw(0, 0) circle (1mm);
		\draw[line width=0.8pt](0.02, 0.09) -- (0.5, 1.95);
		\draw[fill=black](0.53, 2.04) circle (1mm);
		\draw[line width=0.8pt](0.07, 0.07) -- (1.33, 1.5);
		\draw[fill=black](1.39, 1.57) circle (1mm);
		\draw[line width=0.8pt](0.1, 0.04) -- (1.85, 0.78);
		\draw[fill=black](1.93, 0.815) circle (1mm);
		\draw[line width=0.8pt](0.08, -0.04) -- (1.75, -1);
		\draw[fill=black](1.825, -1.045) circle (1mm);
		\draw[line width=0.8pt](0.05, -0.09) -- (1.02, -1.73);
		\draw[fill=black](1.07, -1.8) circle (1mm);
		\draw[fill=black](1.45, 0.2) circle (0.3mm);
		\draw[fill=black](1.45, 0.05) circle (0.3mm);
		\draw[fill=black](1.45, -0.1) circle (0.3mm);
		\draw node (no1) at (-0.3,0) {$0$};
		\draw node (no1) at (0.19, 1.7) {$e_1$};
		\draw node (no1) at (0.86, 1.35) {$e_2$};
		\draw node (no1) at (1.35, 0.8) {$e_3$};
		\draw node (no1) at (1.7, -0.6) {$e_{_{N-1}}$};
		\draw node (no1) at (1.06, -1.25) {$e_{_N}$};
		\draw node (no1) at (0.95, 2.05) {\textbf{Ne}};
		\draw node (no1) at (1.79, 1.57) {\textbf{Ne}};
		\draw node (no1) at (2.3, 0.815) {\textbf{Di}};
		\draw node (no1) at (2.2, -1.045) {\textbf{Di}};
		\draw node (no1) at (1.46, -1.8) {\textbf{Di}};
		\draw(3.95,-0.5) circle (1mm);
		\draw node (no1) at (6.4,-0.5) {{\footnotesize uncontrolled node}\ \ \ \ \ \ \ \ \ \ \ \ \ \ \ \ };
		\draw[fill=black](3.95,-1) circle (1mm);
		\draw node (no1) at (6.4,-1) {\textbf{Ne} {\footnotesize node with Neumann control}};
		\draw[fill=black](3.95,-1.5) circle (1mm);
		\draw node (no1) at (6.4,-1.5) {\textbf{Di} {\footnotesize node with Dirichlet control}\ \ \ };
	\end{tikzpicture}
	\caption{Network with $N$ edges for $m=2$.}	
    \label{fig:m=2}
\end{figure}
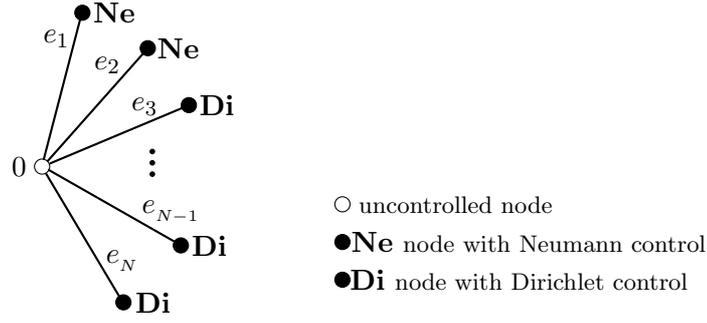

In the full Neumann case, namely, $N\geq 2$ and $m=N$, illustrated in Figure~\ref{fig:fN}, we can prove the controllability if $L\notin\mathcal{N}\cup\mathcal{N}^{*}$.

\begin{figure}[H]
	\centering
	\begin{tikzpicture}
		\draw(0, 0) circle (1mm);
		\draw[line width=0.8pt](0.02, 0.09) -- (0.5, 1.95);
		\draw[fill=black](0.53, 2.04) circle (1mm);
		\draw[line width=0.8pt](0.07, 0.07) -- (1.33, 1.5);
		\draw[fill=black](1.39, 1.57) circle (1mm);
		\draw[line width=0.8pt](0.1, 0.04) -- (1.85, 0.78);
		\draw[fill=black](1.93, 0.815) circle (1mm);
		\draw[line width=0.8pt](0.08, -0.04) -- (1.75, -1);
		\draw[fill=black](1.825, -1.045) circle (1mm);
		\draw[line width=0.8pt](0.05, -0.09) -- (1.02, -1.73);
		\draw[fill=black](1.07, -1.8) circle (1mm);
		\draw[fill=black](1.45, 0.2) circle (0.3mm);
		\draw[fill=black](1.45, 0.05) circle (0.3mm);
		\draw[fill=black](1.45, -0.1) circle (0.3mm);
		\draw node (no1) at (-0.3,0) {$0$};
		\draw node (no1) at (0.19, 1.7) {$e_1$};
		\draw node (no1) at (0.86, 1.35) {$e_2$};
		\draw node (no1) at (1.35, 0.8) {$e_3$};
		\draw node (no1) at (1.7, -0.6) {$e_{_{N-1}}$};
		\draw node (no1) at (1.06, -1.25) {$e_{_N}$};
		\draw node (no1) at (0.95, 2.05) {\textbf{Ne}};
		\draw node (no1) at (1.79, 1.57) {\textbf{Ne}};
		\draw node (no1) at (2.3, 0.815) {\textbf{Ne}};
		\draw node (no1) at (2.2, -1.045) {\textbf{Ne}};
		\draw node (no1) at (1.46, -1.8) {\textbf{Ne}};
		\draw(3.95,-0.5) circle (1mm);
		\draw node (no1) at (6.4,-0.5) {{\footnotesize uncontrolled node}\ \ \ \ \ \ \ \ \ \ \ \ \ \ \ \ };
		\draw[fill=black](3.95,-1) circle (1mm);
		\draw node (no1) at (6.4,-1) {\textbf{Ne} {\footnotesize node with Neumann control}};
	\end{tikzpicture}
	\caption{Network with $N$ edges with control on Neumann condition.}	
    \label{fig:fN}
\end{figure}

Finally, in the full Dirichlet case, that is, $N\geq 2$ and $m=0$, illustrated in Figure~\ref{fig:fD}, we can prove the controllability if $L\notin\mathcal{N}^{*}\cup\mathcal{N}^{\dagger}$, where $\mathcal{N}^{\dagger}$ is defined by \eqref{critical-3}.

\begin{figure}[H]
	\centering
	\begin{tikzpicture}
		\draw(0, 0) circle (1mm);
		\draw[line width=0.8pt](0.02, 0.09) -- (0.5, 1.95);
		\draw[fill=black](0.53, 2.04) circle (1mm);
		\draw[line width=0.8pt](0.07, 0.07) -- (1.33, 1.5);
		\draw[fill=black](1.39, 1.57) circle (1mm);
		\draw[line width=0.8pt](0.1, 0.04) -- (1.85, 0.78);
		\draw[fill=black](1.93, 0.815) circle (1mm);
		\draw[line width=0.8pt](0.08, -0.04) -- (1.75, -1);
		\draw[fill=black](1.825, -1.045) circle (1mm);
		\draw[line width=0.8pt](0.05, -0.09) -- (1.02, -1.73);
		\draw[fill=black](1.07, -1.8) circle (1mm);
		\draw[fill=black](1.45, 0.2) circle (0.3mm);
		\draw[fill=black](1.45, 0.05) circle (0.3mm);
		\draw[fill=black](1.45, -0.1) circle (0.3mm);
		\draw node (no1) at (-0.3,0) {$0$};
		\draw node (no1) at (0.19, 1.7) {$e_1$};
		\draw node (no1) at (0.86, 1.35) {$e_2$};
		\draw node (no1) at (1.35, 0.8) {$e_3$};
		\draw node (no1) at (1.7, -0.6) {$e_{_{N-1}}$};
		\draw node (no1) at (1.06, -1.25) {$e_{_N}$};
		\draw node (no1) at (0.95, 2.05) {\textbf{Di}};
		\draw node (no1) at (1.79, 1.57) {\textbf{Di}};
		\draw node (no1) at (2.3, 0.815) {\textbf{Di}};
		\draw node (no1) at (2.2, -1.045) {\textbf{Di}};
		\draw node (no1) at (1.46, -1.8) {\textbf{Di}};
		\draw(3.95,-0.5) circle (1mm);
		\draw node (no1) at (6.4,-0.5) {{\footnotesize uncontrolled node}\ \ \ \ \ \ \ \ \ \ \ \ \ \ \ \ };
		\draw[fill=black](3.95,-1) circle (1mm);
		\draw node (no1) at (6.4,-1) {\textbf{Di} {\footnotesize node with Dirichlet control}};
	\end{tikzpicture}
	\caption{Network with $N$ edges with control on Neumann condition.}	
    \label{fig:fD}
\end{figure}

Before presenting our main results, we introduce some notations used throughout this article. From now on, we will consider:
	 \begin{itemize}
	 \item[i.] The vector $u$ is given by $$u=(u_1,...,u_N) \in \mathbb{L}^2(\mathcal{T})=\displaystyle\prod_{j=1}^NL^2(0,l_j)$$
	and final and initial data is \[u^0=(u^0_1,...,u^0_N) \in \mathbb{L}^2(\mathcal{T})\quad\text{and}\quad u^T=(u^T_1,...,u^T_N) \in \mathbb{L}^2(\mathcal{T}).\]
	 \item[ii.] The inner product in $\mathbb{L}^2(\mathcal{T})$ will be given by  $$\left(u,z\right)_{\mathbb{L}^2(\mathcal{T})}=\sum_{j=1}^N\int_0^{l_j}u_jz_jdx,\ \ u,z\in \mathbb{L}^2(\mathcal{T}).$$
	Moreover,  $$\mathbb{H}^s(\mathcal{T})=\displaystyle\prod_{j=1}^NH^s(0,l_j), \ \ s \in \mathbb{R}.$$
	 \end{itemize}
	\begin{itemize}
		\item[iii.] Define also the following spaces:
		$$\mathbb{H}_0^k(\mathcal{T})=\displaystyle\prod_{j=1}^{N}H_0^k(0,l_j),\ k\in \mathbb{N}$$
		and
		$$H_r^k(0,l_j)=\left\{v \in H^k(0,l_j),v^{(i-1)}(l_j)=0,\ 1\leq i\leq k\right\},\ k\in \mathbb{N},$$ 
		where $v^{(m)}=\frac{dv}{dx^m}$ and the index $r$ is related to the null right boundary conditions. In addition, 
		$$\mathbb{H}^k_r(\mathcal{T})=\prod_{j=1}^{N}H_r^k(0,l_j)\quad \text{and} \quad \|u\|_{\mathbb{H}^k_r(\mathcal{T})}^2=\sum_{j=1}^{n}\|u_j\|_{H^k(0,l_j)}\ k \in \mathbb{N}.$$
		\item[iv.] Consider the following characterization: 
		$$H_r^{-1}(0,l_j)=\left(H_r^{1}(0,l_j)\right)'$$ 
		as the dual space of $H_r(0,l_j)$ with respect to the pivot space, $L^2(0,l_j)$ and $\mathbb{H}_r^{-1}$ denotes the cartesian product of $H_r^{-1}(0,l_j)$.
		\item[v.] Let
		$$\mathbb{H}^k_e(\mathcal{T})=\left\{u=(u_1,...,u_N)\in \mathbb{H}^k_r(\mathcal{T});\ u_1(0)=u_j(0),\ j=1,\dots,N\right\}\,, \  k \in \mathbb{N},$$
		and
$$\mathbb{B}=C([0,T],\mathbb{L}^2(\mathcal{T}))\cap L^2(0,T,\mathbb{H}^1_e(\mathcal{T})),$$
with the norm 
$$\displaystyle\|u\|_{\mathbb{B}}:=\|u\|_{C([0,T],\mathbb{L}^2(\mathcal{T}))}+\|u\|_{L^2(0,T,\mathbb{H}^1_e(\mathcal{T}))}=\max_{t\in [0,T]}\|u\|_{\mathbb{L}^2(\mathcal{T})}+\left(\int_0^T\|u(t,\cdot)\|_{\mathbb{H}^1_e}^2dt\right)^\frac{1}{2}.$$
\item[vi.] Finally, we need to introduce spaces that are paramount for our main results.  For $m=0,\dots,N$, consider 
\begin{align}\label{xx}
	X_m=\prod_{j=1}^{m}L^2(0,l_j)\times \prod_{i=m+1}^{N}H_0^1(0,l_{i}), \qquad X_0=\prod_{i=1}^{N}H_0^1(0,l_{i}),\qquad X_N=\prod_{j=1}^{N}L^2(0,l_j),
\end{align}
and 
\begin{align}\label{yy}
Y_m=\prod_{j=1}^{m}H^1(0,l_j)\times \prod_{i=m+1}^{N}H^2(0,l_{i}), \qquad Y_0=\prod_{i=1}^{N}H^2(0,l_{i}),\qquad Y_N=\prod_{i=1}^{N}H^{1}(0,l_j).
\end{align}
endowed with the usual Hilbertian norms.	
	\end{itemize}

\subsection{Main result and heuristic} With these previous notations in hand, let us now present our first main result. Given an integer number $N\geq 2$ and $m=0,\dots,N$, this work takes in consideration the following control configuration:
	\begin{align*}
	\begin{cases}
		p_j=0,\ j=1,\dots,m \text{ and }g_j=0,\ j=m+1,\dots,N,&\text{ if }1<m<N,\\
		g_j=0,\ j=1,...,N, &\text{ if }m=0,\\
		p_j=0,\ j=1,...,N,&\text{ if }m=N,
		\end{cases}
\end{align*}
resulting in the following control system
\begin{align}\label{LkdV-2}
\begin{cases}
\partial_{t}u_{j}+\partial_{x}u_{j}+\partial_{x}^{3}u_{j}=0,&t \in (0,T),\ x \in (0,l_j),\ j=1,\dots,N,\\
u_j(t,0)=u_1(t,0),&t\in (0,T),\ j=2,\dots,N,\\
\displaystyle\sum_{j=1}^N\partial_{x}^{2}u_{j}(t,0)=-\alpha u_1(t,0),&t \in (0,T),\\		
u_j(t,l_j)=0,\ \ \ \ \ \ \ \ \ \ \partial_{x}u_{j}(t,l_j)=g_j(t),&t\in (0,T),\ j=1,\dots,m,\\
u_j(t,l_j)=p_j(t),\ \ \ \ \  \partial_{x}u_{j}(t,l_j)=0,&t\in (0,T),\ j=m+1,\dots,N,\\
u_j(0,x)=u_j^0(x),&x \in (0,l_j),\ j=1,\dots,N.
\end{cases}
\end{align}
For $m=0$, we omit the fourth equation from the system above, whereas for $m=N$, we omit the fifth equation. Note that the adjoint system associated with \eqref{LkdV-2} is given by
\begin{align}\label{Adj}
\begin{cases}
-\partial_{t}\varphi_{j}-\partial_{x}\varphi_{j}-\partial_{x}^{3}\varphi_{j}=0,&t \in (0,T),\ x \in (0,l_j),\ j=1,\dots,N,\\
\varphi_j(t,0)=\varphi_1(t,0),&t\in (0,T),\ j=2,\dots,N,\\
\displaystyle\sum_{j=1}^N\partial_{x}^{2}\varphi_{j}(t,0)=(\alpha-N)\varphi_1(t,0),&t \in (0,T)\\				
\varphi_j(t,l_j)=\partial_{x}\varphi_{j}(t,0)=0,&t\in (0,T),\ j=1,\dots,N,\\
\varphi_j(T,x)=\varphi_j^T(x),&x \in (0,l_j),\ j=1,\dots,N.
\end{cases}
\end{align}
Thus, as usual in the control theory (see, for instance, \cite{Lions}), the controllability result for the system \eqref{LkdV-2} is equivalent to proving the following \textit{observability inequality}
\[\|\varphi^T\|_{X_m}^2\leq \tilde{C}\left(\sum_{j=1}^{m}\|\partial_{x}\varphi_{j}(\cdot,l_j)\|_{L^2(0,T)}^2+\sum_{j=m+1}^{N}\|\partial_{x}^{2}\varphi_{j}(\cdot,l_j)\|_{L^2(0,T)}^2\right),\ \forall\varphi^T \in X_m,\]
for the solution of the adjoint system \eqref{Adj}. The first main result of this work is given below and mixes the Neumann and Dirichlet boundary conditions (see Figs. 1, 2, 3, 4, 5, and 6). 

\begin{theorem}\label{main-a}
Let $T>0$ and $u^0,u^T\in \mathbb{L}^2(\mathcal{T})$. Consider $\alpha=N$, $m=0,\dots,N$ and $l_j=L$ for $j=1,\dots,N$.
\begin{enumerate}
\item If $N=2$, then for any $L>0$, there exist controls $(g_1,0)\in [L^2(0,T)]^2$ and $(0,p_2)\in [L^2(0,T)]^2$ such that the unique solution $u$ of \eqref{LkdV-2} satisfies \eqref{ControlCon}.	
\item If $N\geq 3$ and $m=1$ then, there exist controls $(g_1,0,...,0)\in [L^2(0,T)]^N$ and $(0,p_2,...,p_N)\in [L^2(0,T)]^N$ such that the unique solution $u$ of \eqref{LkdV-2} satisfies \eqref{ControlCon}, if and only if $L\notin\mathcal{N}^*$.	
\item If $N\geq 3$ and $m=N-1$ then, there exist controls $(g_1,...,g_{N-1},0)\in [L^2(0,T)]^N$ and $(0,...,0,p_N)\in [L^2(0,T)]^N$ such that the unique solution $u$ of \eqref{LkdV-2} satisfies \eqref{ControlCon}, if and only if $L\notin\mathcal{N}$.
\item If $N>3$ and $1<m<N-1$ then, there exist controls $(g_1,...,g_m,0,...,0)\in [L^2(0,T)]^N$ and $(0,...,0,p_{m+1},...,p_N)\in [L^2(0,T)]^N$ such that the unique solution $u$ of \eqref{LkdV-2} satisfies \eqref{ControlCon}, if and only if $L\notin\mathcal{N}\cup\mathcal{N}^{*}$.
\item If, $N\geq2$, then there exist controls $(g_1,...,g_N)\in [L^2(0,T)]^N$ such that the unique solution $u$ of \eqref{LkdV-2} satisfies \eqref{ControlCon}, if and only if $L\notin\mathcal{N}\cup\mathcal{N}^{*}$.
\item If, $N\geq2$, then there exist controls $(p_1,...,p_N)\in [L^2(0,T)]^N$ such that the unique solution $u$ of \eqref{LkdV-2} satisfies \eqref{ControlCon}, if and only if $L\notin\mathcal{N}^{*}\cup\mathcal{N}^{\dagger}$.
\end{enumerate}
\end{theorem}
For a better description, we can summarize our results in the following table.

\begin{table}[ht]
\begin{tabular}{|l|l|l|l|}
\hline
          & Neumann & Dirichlet & Critical lengths  \\ \hline
$N=2$     & $1$     & $1$       & $\emptyset$       \\ \hline
$N\geq 3$ & $1$     & $N-1$     & $\mathcal{N}^{*}$ \\ \hline
$N\geq 3$ & $N-1$   & $1$       & $\mathcal{N}$     \\ \hline
$N>3$     & $m$     & $N-m$     & $\mathcal{N}^{*}\cup \mathcal{N}$ \\ \hline
$N\geq2$     & $N$     & $0$       & $\mathcal{N}^{*}\cup \mathcal{N}$ \\ \hline
$N\geq2$     & $0$     & $N$       & $\mathcal{N}^{*}\cup \mathcal{N}^{\dagger}$ \\ \hline
\end{tabular}
\caption{Relation between the placement of the controls and the critical lengths.}
\end{table}
We will provide a brief overview of our approach. To establish well-posedness, the classical semigroup theory typically used for the KdV equation in networks, as seen in various studies \cite{Ammari and Crepeau 2018, Cerpa, Parada2022a, Parada2022b}, proves to be less effective for handling Dirichlet controls. Instead, studies like \cite{CaGaPa, Carreno, GG1, Guilleron} mainly rely on ``solutions by transposition" combined with interpolation methods. However, these techniques are still underdeveloped for the KdV equation on networks, where the network structure introduces challenging trace terms in the computations. These terms are addressed in a recent work by the second author \cite{Parada 2023} and in this article.

For the control problem, it is well known in the literature \cite{Lions} that proving the exact controllability of the system \eqref{2} is equivalent to establishing an inequality of \textit{ observability} for the linearized system associated with \eqref{2}. To achieve this, we generally employ the multiplier method and compactness arguments, which reduce the task to demonstrating a unique continuation property for the state operator.

In our context, proving this unique continuation property involves analyzing the spectral problem associated with the relevant linear operator. Specifically, after applying the Fourier transform, the question becomes one of determining when a certain quotient of entire functions remains entire. We introduce a polynomial function \( q: \mathbb{C} \to \mathbb{C} \) and a family of functions
\[
N_{\alpha}: \mathbb{C} \times (0, \infty) \to \mathbb{C},
\]
where \( \alpha \in \mathbb{C} \setminus \{0\} \) and each restriction \( N_{\alpha}(\cdot, L) \) is entire for any \( L > 0 \). We then define a family of functions \( f_{\alpha}(\cdot, L) \) by
\[
f_{\alpha}(\mu, L) = \frac{N_{\alpha}(\mu, L)}{q(\mu)},
\]
within its maximal domain. The problem then reduces to finding \( L > 0 \) such that there exists \( \alpha \in \mathbb{C} \setminus \{0\} \) for which \( f_{\alpha}(\cdot, L) \) is entire. In some works, such as \cite{Capistrano 2017, CaPaRo, Rosier} and references therein, this approach provides an explicit characterization of the set of critical lengths when it exists. However, in many cases, this set cannot be explicitly determined \cite{ArCaDo, CaChSoGo, CaGaPa, CeRiZh}. 

In this context, this work makes two contributions. First, we examine the well-posedness of system \eqref{LkdV-2} under mixed boundary conditions, specifically focusing on Dirichlet and Neumann conditions. Second, we provide, for the first time, a detailed analysis of a star graph structure where the function \( f_{\alpha} \), relevant to our scenarios, is entire. This analysis leads to the characterization of critical sets \(\mathcal{N}\) defined by \eqref{critical}, \(\mathcal{N}^*\) defined by \eqref{critical-2}, and their union \(\mathcal{N} \cup \mathcal{N}^*\), akin to those introduced by Rosier \cite{Rosier} and Glass and Guerrero \cite{GG1}. Additionally, we identify a new critical set, denoted as \(\mathcal{N}^{\dagger}\). Due to the increased complexity of the functions \( N_{\alpha} \) in our setting compared to these earlier works, our approach involves a refined and meticulous adaptation of the analysis of these functions.

\subsection{Outline}Our work is organized as follows: In Section \ref{sec2}, we conduct a detailed analysis of the well-posedness of problem \eqref{LkdV}, incorporating relevant results from the single KdV equation. Building on this, the objective of Section \ref{sec3} is to establish the observability inequality for system \eqref{LkdV-2} with boundary controls acting under mixed boundary conditions, specifically Dirichlet and Neumann conditions, as well as under full  Neumann boundary conditions and full Dirichlet boundary conditions. Here, we carefully examine the function $f_{\alpha}$ associated with our setting, leading to Theorem \ref{main-a}. Finally,  Appendix \ref{trace-review} provides a review of trace estimates for the single KdV equation, while Appendix \ref{control-ND} demonstrates how the observability inequality ensures controllability, and Appendix \ref{app: countable ciritcal} explores various properties of the set $\mathcal{N}^{\dagger}$.

\section{Well-posedness analysis}\label{sec2}

In this section, we examine the well-posedness of problem \eqref{LkdV}, as established in \cite{Parada 2023}, along with its corresponding adjoint problem. To achieve this, we verify the Kato smoothing property for the solution traces, which is essential for deriving an observability inequality that characterizes critical lengths. Additionally, we investigate the regularity of solutions to the adjoint system when the data belong to regular spaces. Note that the positive constants \( C \) introduced here may vary from line to line.

\subsection{Weak solutions} The next proposition was proved in \cite{Parada 2023}, and it gives the well-posedness to the problem
\begin{align}\label{SLAdj}
\begin{cases}
-\partial_{t}v_{j}-\partial_{x}v_{j}-\partial_{x}^{3}v_{j}=f_j,&t \in (0,T),\ x \in (0,l_j),\ j=1,\dots,N,\\
v_j(t,0)=v_1(t,0),&t\in (0,T),\  j=2,\dots,N,\\
\displaystyle\sum_{j=1}^N\partial_{x}^{2}v_{j}(t,0)=(\alpha-N)v_1(t,0),&t \in (0,T),\\	
v_j(t,l_j)=\partial_{x}v_{j}(t,0)=0,&t\in (0,T),\  j=1,\dots,N,\\
v_j(T,x)=0,&x \in (0,l_j),\  j=1,\dots,N,
\end{cases}
\end{align}
for $f=(f_1,...,f_N) \in L^2(0,T,\mathbb{L}^2(\mathcal{T}))$. This result is useful to establish a notion of solution for \eqref{LkdV}.
\begin{proposition} 	If $f \in L^2(0,T,\mathbb{L}^2(\mathcal{T}))$  there exists a unique solution $$v \in \mathbb{B}_1:=C([0,T],\mathbb{H}^1_r(\mathcal{T}))\cap L^2(0,T,\mathbb{H}^2(\mathcal{T}))$$ of \eqref{SLAdj}. Moreover, there exists $C_1>0$ such that
	\begin{align*}
		\|v\|_{\mathbb{B}_1}+\sum_{j=1}^{N}\|\partial_{x}^{2}v_{j}(\cdot,l_j)\|_{L^2(0,T)}+\sum_{j=1}^{N}\|\partial_{x}v_{j}(\cdot,l_j)\|_{H^\frac{1}{3}(0,T)}\leq C_1\|f\|_{L^2(0,T,\mathbb{L}^2(\mathcal{T})}.
	\end{align*}
\end{proposition}

Assume $u_j,p_j,g_j,v_j,f_j\in C^\infty([0,T]\times [0,l_j])$ and $u_j^0\in C^\infty[0,l_j]$. Suppose that $u=(u_1,...,u_N)$ is a solution of \eqref{LkdV} subject to the data $p_j,g_j,u_0^j$ and $v=(v_1,...,v_N)$ is a solution of \eqref{SLAdj} subject to $f=(f_1,...,f_N)$. Multiplying \eqref{SLAdj} by $u_j$, integrating by parts and using the boundary conditions in \eqref{LkdV} and \eqref{SLAdj} we obtain
\begin{align*}
	\sum_{j=1}^N\int_0^{l_j}u_j^0(x)v_j(0,x)-\sum_{j=1}^{N}\int_0^Tp_j(t)\partial_{x}^{2}v_{j}(t,l_j)+\sum_{j=1}^{N}\int_0^Tg_j(t)\partial_{x}v_{j}(t,l_j)=\sum_{j=1}^{N}\int_0^T\int_0^{l_j}f_ju_j.
\end{align*}
Identifying $u_j^0 \in C^\infty([0,l_j])\subset L^2(0,l_j)$ with the functional $(u_j^0)^* \in H_r^{-1}(0,l_j)$ given by
\begin{align*}
	(u_j^0)^*(w)=\int_0^{l_j}u_j^0w,\ \ \forall w \in H_r^{-1}(0,l_j),
\end{align*}
we can write
\begin{align*}
	\int_0^{l_j}u_j^0(x)v_j(0,x)=\langle u_j^0,v_j(0,\cdot)\rangle_{H_r^{-1}(0,l_j)\times H_r^{1}(0,l_j)}.
\end{align*}
Now identifying $g_j \in C^\infty([0,T])\subset L^2(0,T)$ with the functional $g_j^*\in H^{-\frac{1}{3}}(0,T)$ given by
\begin{align*}
	g_j^*(h)=\int_0^Tg_jh,\ \ \forall h \in H^\frac{1}{3}(0,T),
\end{align*}	
we can write
\begin{align*}
	\int_0^Tg_j(t)\partial_{x}v_{j}(t,l_j)=\langle g_j,\partial_{x}v_{j}(\cdot,l_j)\rangle_{H^{-\frac{1}{3}}(0,T)\times H^\frac{1}{3}(0,T)}.
\end{align*}
Therefore we have
\begin{align*}
	\sum_{j=1}^{N}\int_0^T\int_0^{l_j}f_ju_j=&\sum_{j=1}^N\langle u_j^0,v_j(0,\cdot)\rangle_{H_r^{-1}(0,l_j)\times H_r^{1}(0,l_j)}-\sum_{j=1}^{N}\int_0^Tp_j(t)\partial_{x}^{2}v_{j}(t,l_j)\\
	&+\sum_{j=1}^{N}\langle g_j,\partial_{x}v_{j}(\cdot,l_j)\rangle_{H^{-\frac{1}{3}}(0,T)\times H^\frac{1}{3}(0,T)}.
\end{align*}
This motivates the next definition, giving us a weak solution for the system \eqref{LkdV}.
\begin{definition}
	Given $T>0$, $u_0=(u_1^0,u_2^0,\cdots,u_N^0) \in \mathbb{H}_r^{-1}(\mathcal{T})$, $p\in \left(L^2(0,T)\right)^N$ and $g \in (H^{-\frac{1}{3}}(0,T))^N$ a solution by transposition of \eqref{LkdV} is a function $u \in L^2(0,T, \mathbb{L}^2(\mathcal{T}))$ satisfying
	\begin{align*}
		\sum_{j=1}^{N}\int_0^T\int_0^{l_j}f_ju_j=&\sum_{j=1}^N\langle u_j^0,v_j(0,\cdot)\rangle_{H_r^{-1}(0,l_j)\times H_r^{1}(0,l_j)}-\sum_{j=1}^{N}\int_0^Tp_j(t)\partial_{x}^{2}v_{j}(t,l_j)\\
		&+\sum_{j=1}^{N}\langle g_j,\partial_{x}v_{j}(\cdot,l_j)\rangle_{H^{-\frac{1}{3}}(0,T)\times H^\frac{1}{3}(0,T)},
	\end{align*}
	for all $f \in L^2(0,T, \mathbb{L}^2(\mathcal{T}))$, where $v$ is the solution of \eqref{SLAdj} corresponding to $f$. 
\end{definition}
The next result provides the well-posedness of \eqref{LkdV}, proved in \cite{Parada 2023}.
\begin{theorem}
	Let $T>0$ be given. For all $u^0 \in \mathbb{H}_r^{-1}(\mathcal{T})$, $p\in \left(L^2(0,T)\right)^N$ and $g \in (H^{-\frac{1}{3}}(0,T))^N$, there exists a unique solution by transposition for \eqref{LkdV}.
\end{theorem}

\subsection{The adjoint system}  The adjoint system associated with \eqref{LkdV} is given by \eqref{Adj}. The next result ensures that the system \eqref{Adj} admits a unique solution.
\begin{proposition}\label{reg. Adj.}
	For any $\varphi^T \in \mathbb{L}^2(\mathcal{T})$ the system \eqref{Adj} admits a unique solution $\varphi \in \mathbb{B}$ which satisfies
	\begin{equation}
	\|\varphi\|_{\mathbb{B}}\leq C\|\varphi^T\|_{\mathbb{L}^2(\mathcal{T})},\label{est. Adj.-1}
	\end{equation}
	for $C>0$,  and
	\begin{equation}
		\sum_{j=1}^N\|\partial_{x}\varphi_{j}(\cdot,l_j)\|_{L^2(0,T)}^2\leq\|\varphi^T\|_{\mathbb{L}^2(\mathcal{T})}^2\label{est. Adj.-2},
	\end{equation}
for all $\varphi^T \in \mathbb{L}^2(\mathcal{T})$.
\end{proposition}

\begin{proof}
	Define the operator $Au=-\partial_{x}u-\partial_{x}^{3}u$ with domain
	\begin{align*}
		\mathcal{D}(A)=\left\{u \in \mathbb{H}^3(\mathcal{T})\cap \mathbb{H}^1_e(\mathcal{T}),\ \sum_{j=1}^N\partial_{x}^{2}u_{j}(0)=-\alpha u_1(0)\right\},
	\end{align*}
 where
	\begin{align*}
		\partial_x^ku=\left(\partial_x^ku_1,...,\partial_x^ku_N\right).
	\end{align*}
	Observe that $A$ is the operator associated with \eqref{LkdV} when $p_j=g_j=0$. The adjoint operator of $A$ is given by $A^*v=\partial_{x}v+\partial_{x}^{3}v$ with domain
	\begin{align*}
		\mathcal{D}(A^*)=\left\{v \in \mathbb{H}^3(\mathcal{T})\cap \mathbb{H}^1_e(\mathcal{T});\ \partial_{x}v_{j}(0)=0,\  j=1,\dots,N,\ \ \sum_{j=1}^N\partial_{x}^{2}v_{j}(0)=(\alpha-N)v_1(0)\right\}.
	\end{align*}
	$A^*$ is the operator associated with \eqref{Adj}. According to \cite{Ammari and Crepeau 2018}, $A$ is closed, $A$ and $A^*$ are dissipative, so $A^*$ generates a strongly semigroup of contractions on $\mathbb{L}^2(\mathcal{T})$ which will be denoted by $\{S(t)\}_{t\geq 0}$. From the semigroup theory it follows that, for any $\varphi^T \in \mathbb{L}^2(\mathcal{T})$ the problem \eqref{Adj} has a unique (mild) solution $\varphi=S(T-\cdot)\varphi^T \in C([0,T],\mathbb{L}^2(\mathcal{T}))$. Note that
	\begin{align*}
		\|\varphi\|_{ C([0,T],\mathbb{L}^2(\mathcal{T}))}=\|S(T-\cdot)\varphi^T\|_{ C([0,T],\mathbb{L}^2(\mathcal{T}))}\leq \|\varphi^T\|_{\mathbb{L}^2(\mathcal{T})},\ \forall \varphi^T \in \mathbb{L}^2(\mathcal{T}).
	\end{align*}
	
	To see that $\varphi \in L^2(0,T,\mathbb{H}^1_e(\mathcal{T}))$ and to obtain the estimates \eqref{est. Adj.-1} and \eqref{est. Adj.-2}, first suppose $\varphi^T \in \mathcal{D}(A^*)$. In this case, from semigroup theory, we have $$\varphi \in C^1([0,T],\mathbb{L}^2(\mathcal{T}))\cap C([0,T],\mathcal{D}(A^*)).$$ Assume $q_j \in C^\infty([0,T]\times [0,l_j])$ with $q_j(t,0)=q_1(t,0)$. Multiplying the first equation of \eqref{Adj} by $\varphi_j$, integrating by parts, and using the boundary conditions, we obtain
%
	\begin{align*}
		\frac{1}{2}\sum_{j=1}^N\int_0^{l_j}\varphi_j^2(T,x)&=\left(\alpha-\frac{N}{2}\right)\int_0^T\varphi_1^2(t,0)+\frac{1}{2}\sum_{j=1}^N\int_0^T\partial_{x}\varphi_{j}^2(t,l_j)+\frac{1}{2}\sum_{j=1}^N\int_0^{l_j}\varphi_j^2(0,x).
	\end{align*}
Thus, 
	\begin{align*}
		\sum_{j=1}^{N}\|\partial_{x}\varphi_{j}(\cdot,l_j)\|_{L^2(0,T)}^2\leq \|\varphi^T\|_{\mathbb{L}^2(\mathcal{T})}^2,
	\end{align*}
	that is, the map
	\begin{align*}
		\begin{array}{rcl}
			\delta:\mathcal{D}(A^*):&\rightarrow &L^2(0,T)\\
			\varphi^T&\mapsto&\partial_{x}\varphi(\cdot,l_j)
		\end{array}
	\end{align*}
	is continuous. Now, choosing $q_j=x$ it follows that
	\begin{align*}
		\int_0^T\|\varphi_x(t,\cdot)\|_{\mathbb{L}^2(\mathcal{T})}^2\leq \left(\frac{T+2M}{3}\right)\|\varphi^T\|_{\mathbb{L}^2(\mathcal{T})}^2,
	\end{align*}
	where $M=\displaystyle\max_{1\leq j\leq N}l_j$. Then, \eqref{est. Adj.-1} holds for every $\varphi^T\in \mathcal{D}(A^*)$ and consequently the map
	\begin{align*}
		\begin{array}{rcl}
			\Gamma^*:(\mathcal{D}(A^*),\|\cdot\|_{\mathbb{L}^2(\mathcal{T})})&\rightarrow&(\mathbb{B},\|\cdot\|_{\mathbb{B}})\\
			\varphi^T&\mapsto&\Gamma^* \varphi^T=\varphi=S(T-\cdot)\varphi^T
		\end{array}
	\end{align*}
	is continuous. By an argument of density we extend the maps $\Gamma^*$ and $\delta$ to $(\mathbb{L}^2(\mathcal{T}),\|\cdot\|_{\mathbb{L}^2(\mathcal{T})})$. For each $\varphi^T \in \mathbb{L}^2(\mathcal{T})$ we refer to $\delta(\varphi^T)$ when write $\partial_{x}\varphi(t,l_j)$ and, in this sense, \eqref{est. Adj.-2} holds for every $\varphi^T \in \mathbb{L}^2(\mathcal{T})$. Moreover, since
	\begin{align*}
		\Gamma\varphi^T=S(T-\cdot)\varphi^T\ \ \forall\varphi^T \in \mathcal{D}(A^*)
	\end{align*}
	by density and continuity, it follows that
	\begin{align*}
		\Gamma\varphi^T=S(T-\cdot)\varphi^T,\ \ \forall\varphi^T \in \mathbb{L}^2(\mathcal{T})
	\end{align*}
	and thus \eqref{est. Adj.-1} was verified.
\end{proof}

\subsection{Trace estimates} Now, we proceed to prove that the solutions of \eqref{Adj} possess the sharp Kato smoothing property, namely
\begin{align*}
	\partial_x^k\varphi_j\in L_x^\infty\left(0,l_j;H^{\frac{1-k}{3}}(0,T)\right),\ k=0,1,2,\ \ j=1,\dots,N.
\end{align*}
The first result ensures the hidden regularity.
\begin{proposition}\label{trace-Adj.}
	For $T>0$ and every $\varphi^T\in \mathbb{L}^2(\mathcal{T})$, the solution $\varphi$ of \eqref{Adj} posses the hidden regularity
	\begin{align*}
		\partial_x^k\varphi_j\in L_x^\infty\left(0,l_j;H^{\frac{1-k}{3}}(0,T)\right),\ k=0,1,2,\ \ j=1,\dots,N.
	\end{align*}
	Moreover, there exists $C>0$ such that
	\begin{align*}
		\sum_{k=0}^2\sup_{x \in (0,L)}\|\partial_x^k\varphi_j(\cdot,x)\|_{H^\frac{1-k}{3}(0,T)}\leq C\|\varphi^T\|_{\mathbb{L}^2(\mathcal{T})},\ j=1,\dots,N.
	\end{align*}
\end{proposition}
\begin{proof}  Our proof will be split into two cases.
We divide it into two cases.

\vspace{0.2cm}
\noindent\textbf{Case 1:} $N \in \mathbb{N}$ and $l_j=L$, for every $j=1,\dots,N$.
\vspace{0.2cm}

For each $j=1,\dots,N$ define $\xi_j=\varphi_1-\varphi_j$. Note that $$\xi_j \in \mathcal{Z}_T=C([0,T];L^2(0,L))\cap L^2(0,T;H^1(0,L))$$ and its solves
\begin{align*}
	\begin{cases}
		\partial_{t}\xi_{j}+\partial_{x}\xi_{j}+\partial_{x}^{3}\xi_{j}=0,&t\in(0,T),\ x \in (0,L),\\
		\xi_j(t,0)=\xi_j(t,L)=\partial_{x}\xi_{j}(t,0)=0,&t \in(0,T),\\
		\xi_j(T,x)=\varphi_1^T(x)-\varphi_j^T(x),&x \in (0,L).
\end{cases}
\end{align*}
Thanks to Corollary \ref{W.P.KdV(0,L)-4} we get $\partial_x^k\xi_j\in L_x^\infty\left(0,L;H^{\frac{1-k}{3}}(0,T)\right)$, for $k=0,1,2$, and
\begin{align}\label{eq 12-b}
	\|\xi_j\|_{\mathcal{Z}_T}+\sum_{k=0}^2\sup_{x \in (0,L)}\|\partial_x^k\xi_j(\cdot,x)\|_{H^\frac{1-k}{3}(0,T)}\leq C\left(\|\varphi_1^T-\varphi_j^T\|_{L^2(0,L)}\right)\leq C\|\varphi^T\|_{\mathbb{L}^2(\mathcal{T})}
\end{align}
where $C>0$ is a constant.

Now, since $\varphi \in L^2(0,T;\mathbb{H}_e^1(\mathcal{T}))$ and $H^1(0,L)\hookrightarrow C([0,L])$
\begin{align*}
	\int_0^T|\varphi_1(t,0)|^2dt
	&\leq c\|\varphi\|_{L^2(0,T;\mathbb{H}_e^1(\mathcal{T}))}^2
\end{align*}
for some $c>0$. From Proposition \ref{reg. Adj.} It follows that
\begin{align*}
	\int_0^T|\varphi_1(t,0)|^2dt\leq \tilde{C}\|\varphi^T\|_{\mathbb{L}^2(\mathcal{T})}^2,
\end{align*}
for $\tilde{C}>0$. Then $\varphi_1(\cdot,0) \in L^2(0,T)\hookrightarrow H^{-\frac{1}{3}}(0,T)$ (this embedding is an isometry) and by the last inequality
\begin{align}\label{eq 14}
	\|\varphi_1(\cdot,0)\|_{H^{-\frac{1}{3}}(0,T)}=\|\varphi_1(\cdot,0)\|_{L^2(0,T)}\leq \tilde{C}\|\varphi^T\|_{\mathbb{L}^2(\mathcal{T})}.
\end{align}
Defining $\displaystyle\psi=\sum_{j=1}^N\varphi_j$ we conclude that $$\psi \in \mathcal{Z}_T=C([0,T];L^2(0,L))\cap L^2(0,T;H^1(0,L))$$ and it solves the system
\begin{align*}
\begin{cases}
		\partial_{t}\psi+\partial_{x}\psi+\partial_{x}^{3}\psi=0,&t\in(0,T),\ x \in (0,L),\\
		\psi(t,L)=\partial_{x}\psi(t,0)=0,\ \ \partial_{x}^{2}\psi(t,0)=(\alpha-N)\varphi_1(t,0),&t \in(0,T),\\
		\psi(T,x)=\displaystyle\sum_{j=1}^N\varphi_j^T(x),&x \in (0,L).
\end{cases}
\end{align*}
Then, the Proposition \ref{W,P.Kdv(0,L)-5} gives us $\partial_x^k\psi\in L_x^\infty\left(0,L;H^{\frac{1-k}{3}}(0,T)\right)$ for $k=0,1,2$ and
\begin{align*}
	\begin{aligned}
	&\|\psi\|_{\mathcal{Z}_T}+\sum_{k=0}^2\sup_{x \in (0,L)}\|\partial_x^k\psi(\cdot,x)\|_{H^\frac{1-k}{3}(0,T)}\leq C\left(\sum_{j=1}^N\|\varphi_j^T\|_{L^2(0,L)}+|\alpha-N|\|\varphi_1(\cdot,0)\|_{H^{-\frac{1}{3}}(0,T)}\right),
\end{aligned}
\end{align*}
for some positive constant $C$. Using \eqref{eq 14} we get another constant, still denoted by $C>0$ such that
\begin{align}\label{eq 14-b}
	\|\psi\|_{\mathcal{Z}_T}+\sum_{k=0}^2\sup_{x \in (0,L)}\|\partial_x^k\psi(\cdot,x)\|_{H^\frac{1-k}{3}(0,T)}
	\leq C(|\alpha-N|)\|\varphi^T\|_{\mathbb{L}^2(\mathcal{T})}.
\end{align}

Observe that
\begin{align*}
\sum_{j=1}^N\xi_j+\psi&=\sum_{j=1}^N(\varphi_1-\varphi_j)+\sum_{j=1}^N\varphi_j=N\varphi_1,
\end{align*}
that is,
\begin{align*}
\varphi_1=\frac{1}{N}\left(\sum_{j=1}^N\xi_j+\psi\right).
\end{align*}
But we know that $\psi,\xi_j\in L_x^\infty\left(0,L;H^{\frac{1-k}{3}}(0,T)\right)$ for $j=1,\dots,N$, so we conclude that $$\partial_x^k\varphi_1\in L_x^\infty\left(0,L;H^{\frac{1-k}{3}}(0,T)\right),$$ for $k=0,1,2$. Moreover,
\begin{align*}
	\sum_{k=0}^2\sup_{x \in (0,L)}\|\partial_x^k\varphi_1(\cdot,x)\|_{H^\frac{1-k}{3}(0,T)}\leq& \frac{1}{N}\left(\sum_{j=1}^N\sum_{k=0}^2\sup_{x \in (0,L)}\|\partial_x^k\xi_j(\cdot,x)\|_{H^\frac{1-k}{3}(0,T)}+\sum_{k=0}^2\sup_{x \in (0,L)}\|\partial_x^k\psi(\cdot,x)\|_{H^\frac{1-k}{3}(0,T)}\right).
\end{align*}
Due to the inequalities \eqref{eq 12-b} and \eqref{eq 14-b} it follows that
\begin{align}\label{eq 15-b}
	\sum_{k=0}^2\sup_{x \in (0,L)}\|\partial_x^k\varphi_1(\cdot,x)\|_{H^\frac{1-k}{3}(0,T)}&\leq C\|\varphi^T\|_{\mathbb{L}^2(\mathcal{T})}
\end{align}
where $C=C(\alpha,N)>0$. Now, since $\varphi_j$ is the solution of
\begin{align*}
\begin{cases}
		\partial_{t}\varphi_{j}+\partial_{x}\varphi_{j}+\partial_{x}^{3}\varphi_{j}=0,&t \in (0,T),\ x \in (0,L),\\
		\varphi_j(t,0)=\varphi_1(t,0), \ \varphi_j(t,L)=\partial_{x}\varphi_{j}(t,0)=0,&t\in (0,T),\\
		\varphi_j(T,x)=\varphi_j^T(x),&x \in (0,L),
	\end{cases}
\end{align*}
as $\varphi_1(t,0)\in H^{\frac{1}{3}}(0,T)$, Corollary \ref{W.P.KdV(0,L)-4} ensures that $\partial_x^k\varphi_j\in L_x^\infty\left(0,L;H^{\frac{1-k}{3}}(0,T)\right)$ for $k=0,1,2$ and 
using \eqref{eq 15-b} we get
\begin{align*}
	\sum_{k=0}^2\sup_{x \in (0,L)}\|\partial_x^k\varphi_j(\cdot,x)\|_{H^\frac{1-k}{3}(0,T)}	
	&\leq C\|\varphi^T\|_{\mathbb{L}^2(\mathcal{T})},
\end{align*}
for $C>0$ showing this case. 

\vspace{0.2cm}
\noindent\textbf{Case 2:} $N \in \mathbb{N}$ and arbitrary lengths $l_j>0,\ j=1,\dots,N$.
\vspace{0.2cm}

Define $L=\max_{j=1,\dots,N}l_j$ and
\begin{align*}
	\tilde{\varphi}_j^T(x):=
	\begin{cases}
		\varphi_j^T(x),&\ x \in (0,l_j),\\
		0,&x \in (0,L)\backslash(0,l_j).
	\end{cases}
\end{align*}
Denote by $\tilde{\varphi}\in C([0,T],\mathbb{L}^2(\tilde{\mathcal{T}}))\cap L^2(0,T,\mathbb{H}^1_e(\tilde{\mathcal{T}}))$ the solution of \eqref{LkdV-2} associated to data $\tilde{\varphi}^T=(\tilde{\varphi}_1^T,...,\tilde{\varphi}_N^T)\in \mathbb{L}^2(\tilde{\mathcal{T}})=\left(L^2(0,L)\right)^N$. From the previous case, we have
\begin{align*}
	\sum_{k=0}^2\sup_{x \in (0,L)}\|\partial_x^k\tilde{\varphi}_j(\cdot,x)\|_{H^\frac{1-k}{3}(0,T)}	&\leq C\|\tilde{\varphi}^T\|_{\mathbb{L}^2(\tilde{\mathcal{T})}}
\end{align*}
It follows from the definition that $\|\tilde{\varphi}^T\|_{\mathbb{L}^2(\tilde{\mathcal{T})}}=\|\varphi^T\|_{\mathbb{L}^2(\mathcal{T})}$. Also from the definition we have $\tilde{\varphi}_j^T=\varphi_j^T$ in $(0,l_j)$ so $\varphi_j$ and $\tilde{\varphi}_j$ solve the system
\begin{align*}
	\begin{cases}
		\partial_{t}\varphi_{j}+\partial_{x}\varphi_{j}+\partial_{x}^{3}\varphi_{j}=0,&t \in (0,T),\ x \in (0,l_j),\\
		\varphi_j(t,0)=\varphi_1(t,0)=\varphi_j(t,l_j)=\partial_{x}\varphi_{j}(t,0)=0,& t \in (0,T),\\
		\varphi_j(T,x)=\varphi_j^T(x),&x \in (0,l_j).
	\end{cases}
\end{align*}
Then by uniqueness of solution we obtain $\tilde{\varphi}_j=\varphi_j$ in $(0,T)\times (0,l_j)$. Consequently, for $j=1,\dots,N$,
\begin{align*}
\sum_{k=0}^2\sup_{x \in (0,l_j)}\|\partial_x^k\varphi_j(\cdot,x)\|_{H^\frac{1-k}{3}(0,T)}
&\leq \sum_{k=0}^2\sup_{x \in (0,L)}\|\partial_x^k\tilde{\varphi}_j(\cdot,x)\|_{H^\frac{1-k}{3}(0,T)}\leq C\|\tilde{\varphi}^T\|_{\mathbb{L}^2(\tilde{\mathcal{T})}}=C\|\varphi^T\|_{\mathbb{L}^2(\mathcal{T})}.
\end{align*}
which concludes case 2 and, consequently, the proof of Proposition \ref{trace-Adj.}.
\end{proof}

To finish this subsection, we establish that the traces $\partial_{x}^{2}\varphi_{j}(\cdot,0),\partial_{x}^{2}\varphi_{j}(\cdot,l_j)$ belong to $L^2(0,T)$ and depend continuously on the initial data.
\begin{proposition}\label{phi_xx in L^2}
	Given $\varphi^T \in \mathbb{L}^2(\mathcal{T})$,  we have that $\partial_{x}^{2}\varphi_{j}(\cdot,x) \in L^2(0,T)$, for any $x \in [0,l_j]$. Moreover, the following inequality holds
	\begin{align*}
		\int_0^T\left(\partial_{x}^{2}\varphi_{j}(t,x)\right)^2dt\leq C\|\varphi^T\|_{\mathbb{L}^2(\mathcal{T})}^2,\ \forall\ \varphi^T \in \mathbb{L}^2(\mathcal{T}),
	\end{align*}
	for some positive constant $C$.
\end{proposition}
\begin{proof} Consider the operator $$P_j:H^3(0,l_j)\rightarrow L^2(0,l_j)$$ defined by $P_jw=\partial_{x}w+\partial_{x}^{3}w$. The graph norm in $H^3(0,l_j)$ associated to $P_j$ is $$\|w\|_{P_j}=\|w\|_{L^2(0,l_j)}+\|P_jw\|_{L^2(0,l_j)}.$$ From \cite[Lemma A.2]{KdV by flatness} there exists a constant $d_1>0$ such that
	\begin{align}\label{H^3<P_j}
		\|w\|_{H^3(0,l_j)}\leq d_1\|w\|_{P_j},\ \forall w \in H^3(0,l_j),\ j=1,\dots,N.
	\end{align}
	On the other hand, given $v=(v_1,...,v_N) \in \mathcal{D}(A^*)$ we have
	\begin{align*}
		\|v\|_{\mathcal{D}(A^*)}=\|v\|_{\mathbb{L}^2(\mathcal{T})}+\|\partial_{x}v+\partial_{x}^{3}v\|_{\mathbb{L}^2(\mathcal{T})}\geq \|v\|_{\mathbb{L}^2(\mathcal{T})}=\left(\sum_{j=1}^N\|v_j\|_{L^2(0,l_j)}^2\right)^\frac{1}{2}\geq \|v_j\|_{L^2(0,l_j)}
	\end{align*}
	as well as
	\begin{align*}
		\|v\|_{\mathcal{D}(A^*)}\geq \|\partial_{x}v+\partial_{x}^{3}v\|_{\mathbb{L}^2(\mathcal{T})}=\left(\sum_{j=1}^N\|\partial_{x}v_{j}+\partial_{x}^{3}v_{j}\|_{L^2(0,l_j)}^2\right)^\frac{1}{2}\geq \|\partial_{x}v_{j}+\partial_{x}^{3}v_{j}\|_{L^2(0,l_j)}.
	\end{align*}
Therefore, this yields
	\begin{align}\label{P_j<D(A*)}
		\|v_j\|_{P_j}\leq 2\|v\|_{\mathcal{D}(A^*)},\ j=1,\dots,N.
	\end{align}
	
Let us first assume that  $\varphi_j^T \in \mathcal{D}(A^*)$. 

\vspace{0.2cm}
\noindent\textbf{Claim:} We claim that $\partial_{x}^{2}\varphi_{j}\in C([0,T]\times [0,l_j])$. 
\vspace{0.2cm}

Indeed, since $\varphi \in C([0,T],\mathcal{D}(A^*))$ we have $\partial_{x}^{2}\varphi_{j}(t,\cdot)\in H^1(0,l_j)\hookrightarrow C([0,l_j])$. Hence, fixed $x \in [0,l_j]$, for some constant $x>0$ we get
	\begin{align*}
		|\partial_{x}^{2}\varphi_{j}(t,x)-\partial_{x}^{2}\varphi_{j}(t_0,x)|
        &\leq c\|\partial_{x}^{2}\varphi_{j}(t,\cdot)-\partial_{x}^{2}\varphi_{j}(t_0,\cdot)\|_{H^1(0,l_j)}\leq d_2\|\varphi_j(t,\cdot)-\varphi_j(t_0,\cdot)\|_{H^3(0,l_j)}
	\end{align*}
	for any $t,t_0\in [0,T]$. Using \eqref{H^3<P_j} and \eqref{P_j<D(A*)} we obtain
	\begin{align*}
		|\partial_{x}^{2}\varphi_{j}(t,x)-\partial_{x}^{2}\varphi_{j}(t_0,x)|&\leq C\|\varphi_j(t,\cdot)-\varphi_j(t_0,\cdot)\|_{\mathcal{D}(A^*)},
	\end{align*}
for $C>0$ and from where we conclude that $\partial_{x}^{2}\varphi_{j}(\cdot,x)\in C([0,T])$ and, consequently, $\partial_{x}^{2}\varphi_{j}(\cdot,x)\in L^2(0,T)$, showing the claim.  
	
	\vspace{0.1cm }
	Using Proposition \ref{trace-Adj.} and the fact that $L^2(0,T)\hookrightarrow H^{-\frac{1}{3}}(0,T)$ is an isometry we obtain, for each $x \in [0,l_j]$, that
	\begin{align*}
		\int_0^T\left(\partial_{x}^{2}\varphi_{j}(t,x)\right)^2dt=\|\partial_{x}^{2}\varphi_{j}(\cdot,x)\|_{L^2(0,T)}^2=\|\partial_{x}^{2}\varphi_{j}(\cdot,x)\|_{H^{-\frac{1}{3}}(0,T)}^2\leq C\|\varphi_j^T\|_{\mathbb{L}^2(\mathcal{T})}^2.
	\end{align*}
	Finally, with an argument of density and continuity, we extend the result for every $\varphi_j^T \in \mathbb{L}^2(\mathcal{T})$. The result is thus achieved. 
\end{proof}

\subsection{Regularity}  With the spaces \eqref{xx} and \eqref{yy} in hand, we have the following proposition.
\begin{proposition}\label{W.P.L^2xH_0^1}
	If the final data $\varphi^T \in X_m$, then the corresponding solution $\varphi$ of the system \eqref{Adj} has the additional regularity $\varphi \in L^2(0,T;Y_m)$	with
	\begin{align*}
		\|\varphi\|_{L^2(0,T;Y_m)}\leq C\|\varphi^T\|_{X_m}.
	\end{align*}
	for some positive constant $C$. Furthermore, the estimate
	\begin{align*}
		\|\varphi^T\|_{X_m}^2&\leq \frac{1}{T}\sum_{j=1}^N\int_0^T\int_0^{l_j}\varphi_j^2+\left(\frac{1}{T}+\frac{6}{\sigma}\right)\sum_{j=m+1}^N\int_0^T\int_0^{l_j}\partial_{x}\varphi_{j}^2+2\left(\alpha-\frac{N}{2}\right)\int_0^T\varphi_1^2(t,0)\\
		&+2\sum_{j=m+1}^N\|\varphi_j^T\|_{L^2(0,l_j)}^2+\sum_{j=1}^m\int_0^T\partial_{x}\varphi_{j}^2(t,l_j)+\sum_{j=m+1}^N\int_0^T\partial_{x}^{2}\varphi_{j}^2(t,l_j),
	\end{align*}
is verified for every $\varphi^T\in  X_m$, where $\displaystyle\sigma:=\min_{1\leq j\leq N}l_j$.
\end{proposition}
\begin{proof} Once we have that
	\begin{align*}	\|\varphi\|_{L^2(0,T;Y_m)}^2 
	&=\int_0^T\sum_{j=1}^m\|\varphi_j(t,\cdot)\|_{H^1(0,l_j)}^2dt+\int_0^T\sum_{j=m+1}^N\|\varphi_j(t,\cdot)\|_{H^2(0,l_j)}^2dt\\
		&\leq \|\varphi\|_{L^2(0,T;\mathbb{H}_e^1(\mathcal{T}))}^2+\sum_{j=m+1}^N\int_0^T\|\varphi_j(t,\cdot)\|_{H^2(0,l_j)}^2dt,
	\end{align*}
	from the first part of Proposition \ref{reg. Adj.}, it is sufficient to show that $\varphi_j \in L^2(0,T;H^2(0,l_j))$ for $j=m+1,\dots,N$ with
	\begin{align}\label{eq 2.14}
		\sum_{j=m+1}^N\|\varphi_j\|_{L^2(0,T;H^2(0,l_j))}^2\leq \tilde{C}\|\varphi^T\|_{X_m}^2
	\end{align}
	for some constant $\tilde{C}>0$.
	
	To do this, first assume $\varphi^T \in \mathcal{D}(A^*)\cap X_m$, the result for all $\varphi^T \in X_m$ follows by an argument of density. In this case, it is clear that $\varphi_j(t,\cdot)\in H^2(0,l_j)$. Multiplying the adjoint system \eqref{Adj} by $(l_j-x)\partial_{x}^{2}\varphi_{j}$, integrating by parts, and using the boundary conditions, we get
	\begin{align*}
		\frac{1}{2}\int_0^{l_j}(l_j-x)\partial_{x}\varphi_{j}^2(0,x)+\frac{3}{2}\int_0^T\int_0^{l_j}\partial_{x}^{2}\varphi_{j}^2=&\int_0^T(\partial_{x}\varphi_{j}\partial_{x}^{2}\varphi_{j})(t,l_j)+\frac{1}{2}\int_0^T\int_0^{l_j}\partial_{x}\varphi_{j}^2\\
		&+\frac{l_j}{2}\int_0^T\partial_{x}^{2}\varphi_{j}^2(t,0)+\frac{1}{2}\int_0^{l_j}(l_j-x)\partial_{x}\varphi_{j}^2(T,x).
	\end{align*}
	By Young's inequality and Propositions \ref{reg. Adj.} and  \ref{phi_xx in L^2}, we conclude that there exists $C>0$ such that
	\begin{align*}
		\int_0^T(\partial_{x}\varphi_{j}\partial_{x}^{2}\varphi_{j})(t,l_j)&\leq \frac{1}{2}\int_0^T\partial_{x}\varphi_{j}^2(t,l_j)+\frac{1}{2}\int_0^T\partial_{x}^{2}\varphi_{j}^2(t,l_j)\leq C\|\varphi^T\|_{\mathbb{L}^2(\mathcal{T})}^2.
	\end{align*}
	Defining $M=\displaystyle\max_{1\leq j\leq N}l_j$, for $j=m+1,\dots,N$, and, again, using Propositions \ref{reg. Adj.} and  \ref{phi_xx in L^2}, it follows that exists another constant, still denote by $C>0$, such that
	\begin{align*}
		\int_0^T\int_0^{l_j}\partial_{x}^{2}\varphi_{j}^2&
		\leq C\|\varphi^T\|_{X_m}^2.
	\end{align*}
	Combining this with the Proposition \ref{reg. Adj.} we obtain, $\varphi_{j}\in L^2(0,T;H^2(0,l_j)$ and
	\begin{align*}
		\|\varphi_j\|_{L^2(0,T;H^2(0,l_j))}\leq C^*\|\varphi^T\|_{X_m},\ \forall\ \varphi^T\in \mathcal{D}(A^*)\cap X_m,
	\end{align*}
	where $C^*>0$ is a constant and $j=m+1,\dots,N$. Since $\mathcal{D}(A^*)\cap X_m$ is dense in $X_m$, the map
	\begin{align*}
		\begin{array}{rcl}
			\mathcal{D}(A^*)\cap X_m&\rightarrow&L^2(0,T;H^2(0,l_j))\\
			\varphi_T&\mapsto&\varphi_j
		\end{array}
	\end{align*}
	extends continuously to the entire $X_m$ with
	\begin{align*}
		\|\varphi_j\|_{L^2(0,T;H^2(0,l_j))}\leq C^*\|\varphi^T\|_{X_m},\ \forall \varphi^T\in X_m
	\end{align*}
	and therefore \eqref{eq 2.14} holds with $\tilde{C}=(N-m)(C^*)^2$.
	
	\vspace{0.1cm}
	For the second part, multiplying the adjoint system \eqref{Adj} by $t\partial_{x}^{2}\varphi_{j}$, integrating by parts, we get
	\begin{align*}
		\frac{1}{2}\int_0^Tt\partial_{x}^{2}\varphi_{j}^2(t,l_j)+\frac{1}{2}\int_0^Tt\partial_{x}\varphi_{j}^2(t,l_j)+\frac{1}{2}\int_0^T\int_0^{l_j}\partial_{x}\varphi_{j}^2=&\frac{1}{2}\int_0^Tt\partial_{x}^{2}\varphi_{j}^2(t,0)+\frac{T}{2}\int_0^{l_j}\partial_{x}\varphi_{j}^2(T,x),
	\end{align*}
	where it follows that
	\begin{align}\label{eq 2.16}
		\|\partial_{x}\varphi_{j}^T\|_{L^2(0,l_j)}^2	\leq \frac{1}{T}\int_0^T\int_0^{l_j}\partial_{x}\varphi_{j}^2+\int_0^T\partial_{x}\varphi_{j}^2(t,l_j)+\int_0^T\partial_{x}^{2}\varphi_{j}^2(t,l_j),\ j=1,\dots,N.
	\end{align}
	On the other hand, by multiplying the adjoint system \eqref{Adj} by $x\varphi_j$, integrating by parts, and using the boundary conditions, we get 
	\begin{align*}
		\frac{l_j}{2}\int_0^T\partial_{x}\varphi_{j}^2(t,l_j)+\frac{1}{2}\int_0^T\int_0^{l_j}\varphi_j^2+\frac{1}{2}\int_0^{l_j}x\varphi_j^2(0,x)&=\frac{1}{2}\int_0^{l_j}x\varphi_j^2(T,x)+\frac{3}{2}\int_0^T\int_0^{l_j}\partial_{x}\varphi_{j}^2
	\end{align*}
	which implies, in particular, that
	\begin{align}\label{eq 2.18}
		\int_0^T\partial_{x}\varphi_{j}^2(t,l_j)\leq \int_0^{l_j}\varphi_j^2(T,x)+\frac{3}{l_j}\int_0^T\int_0^{l_j}\partial_{x}\varphi_{j}^2\ ,\ j=1,\dots,N.
	\end{align}
	Now, multiplying the adjoint system \eqref{Adj} by $t\varphi_j$, integrating by parts, and using the boundary conditions, we also have that
	\begin{align}\label{eq 2.19}
		\sum_{j=1}^N\|\varphi_j^T\|_{L^2(0,l_j)}^2&\leq \frac{1}{T}\sum_{j=1}^N\int_0^T\int_0^{l_j}\varphi_j^2+2\left(\alpha-\frac{N}{2}\right)\int_0^T\varphi_1^2(t,0)+\sum_{j=1}^N\int_0^T\partial_{x}\varphi_{j}^2(t,l_j),
	\end{align}
	for $j=1,\dots,N$. Using \eqref{eq 2.16} and \eqref{eq 2.19} we get that
	\begin{align*}
		\|\varphi^T\|_{X_m}^2
		\leq &\frac{1}{T}\sum_{j=1}^N\int_0^T\int_0^{l_j}\varphi_j^2+2\left(\alpha-\frac{N}{2}\right)\int_0^T\varphi_1^2(t,0)+\sum_{j=1}^N\int_0^T\partial_{x}\varphi_{j}^2(t,l_j)\\
		&+\frac{1}{T}\sum_{j=m+1}^N\int_0^T\int_0^{l_j}\partial_{x}\varphi_{j}^2+\sum_{j=m+1}^N\int_0^T\partial_{x}\varphi_{j}^2(t,l_j)+\sum_{j=m+1}^N\int_0^T\partial_{x}^{2}\varphi_{j}^2(t,l_j).
	\end{align*}
Observe that from \eqref{eq 2.18}, the following holds
	\begin{align*}
		2\sum_{j=m+1}^N\int_0^T\partial_{x}\varphi_{j}^2(t,l_j)\leq 2\sum_{j=m+1}^N\int_0^{l_j}\varphi_j^2(T,x)+\sum_{j=m+1}^N\frac{6}{l_j}\int_0^T\int_0^{l_j}\partial_{x}\varphi_{j}^2,
	\end{align*}
	and consequently
	\begin{align*}
		\|\varphi^T\|_{X_m}^2\leq& \frac{1}{T}\sum_{j=1}^N\int_0^T\int_0^{l_j}\varphi_j^2+\left(\frac{1}{T}+\frac{6}{\sigma}\right)\sum_{j=m+1}^N\int_0^T\int_0^{l_j}\partial_{x}\varphi_{j}^2+2\left(\alpha-\frac{N}{2}\right)\int_0^T\varphi_1^2(t,0)\\
		&+2\sum_{j=m+1}^N\|\varphi_j^T\|_{L^2(0,l_j)}^2+\sum_{j=1}^m\int_0^T\partial_{x}\varphi_{j}^2(t,l_j)+\sum_{j=m+1}^N\int_0^T\partial_{x}^{2}\varphi_{j}^2(t,l_j),
	\end{align*}
	giving the result. 
	\end{proof}

\section{A study of entire functions}\label{sec3}
From this point, our analysis focuses on establishing the observability inequality. Let us first discuss the observability inequality.
\subsection{Observability inequality} We are here interested in proving the following observability inequality, that is, the existence of a constant $C>0$ such that
\begin{align}\label{Obs. Ineq.}
	\|\varphi^T\|_{\mathbb{L}^2(\mathcal{T})}^2\leq C\left(\sum_{j=1}^{m}\|\partial_{x}\varphi_{j}(\cdot,l_j)\|_{L^2(0,T)}^2+\sum_{j=m+1}^{N}\|\partial_{x}^{2}\varphi_{j}(\cdot,l_j)\|_{L^2(0,T)}^2\right),\ \forall\varphi^T \in \mathbb{L}^2(\mathcal{T}).
\end{align}
Observe that to check \eqref{Obs. Ineq.} it is sufficient to check
\begin{align}\label{Obs. Ineq. 2}
	\|\varphi^T\|_{X_m}^2\leq \tilde{C}\left(\sum_{j=1}^{m}\|\partial_{x}\varphi_{j}(\cdot,l_j)\|_{L^2(0,T)}^2+\sum_{j=m+1}^{N}\|\partial_{x}^{2}\varphi_{j}(\cdot,l_j)\|_{L^2(0,T)}^2\right),\ \forall\varphi^T \in X_m,
\end{align}
for some positive constant $\tilde{C}$.

Indeed, given $\varphi^T \in \mathbb{L}^2(\mathcal{T})$, since $X_m$ is dense in $\mathbb{L}^2(\mathcal{T})$ there exists a sequence $(\varphi^{T,n})_{n \in \mathbb{N}}\subset X_m$ such that $\varphi^{T,n}\rightarrow\varphi^T$ in $\mathbb{L}^2(\mathcal{T})$. From Propositions \ref{reg. Adj.} and \ref{phi_xx in L^2} it follows that $\partial_{x}\varphi_{j}^n(\cdot,l_j)\rightarrow \partial_{x}\varphi_{j}(\cdot,l_j)$ and $\partial_{x}^{2}\varphi_{j}^n(\cdot,l_j)\rightarrow \partial_{x}^{2}\varphi_{j}(\cdot,l_j)$ in $L^2(0,T)$. Furthermore, using the Poincaré inequality, we get, for some $c>0$,
\begin{align*}
	\|\varphi^{T,n}\|_{\mathbb{L}^2(\mathcal{T})}\leq c\|\varphi^{T,n}\|_{X_m}.
\end{align*}
If \eqref{Obs. Ineq. 2} holds, then
\begin{align*}
	\|\varphi^{T,n}\|_{\mathbb{L}^2(\mathcal{T})}^2\leq c^2\tilde{C}\left(\sum_{j=1}^{m}\|\partial_{x}\varphi_{j}^n(\cdot,l_j)\|_{L^2(0,T)}^2+\sum_{j=m+1}^{N}\|\partial_{x}^{2}\varphi_{j}^n(\cdot,l_j)\|_{L^2(0,T)}^2\right)
\end{align*}
and passing to the limit, it follows that
\begin{align*}
	\|\varphi^{T}\|_{\mathbb{L}^2(\mathcal{T})}^2\leq c^2\tilde{C}\left(\sum_{j=1}^{m}\|\partial_{x}\varphi_{j}(\cdot,l_j)\|_{L^2(0,T)}^2+\sum_{j=m+1}^{N}\|\partial_{x}^{2}\varphi_{j}(\cdot,l_j)\|_{L^2(0,T)}^2\right)
\end{align*}
so \eqref{Obs. Ineq.} holds with $C=c^2\tilde{C}>0$. 
\begin{remark} 
The advantage of working with data in $X_m$ for the adjoint system is that, in addition to the additional regularity for the solution, the estimate given in the Proposition \ref{W.P.L^2xH_0^1} carries information related to the two types of controls used, while the corresponding estimate for data in $\mathbb{L}^2(\mathcal{T})$ carries only information related to the Neumann controls types.
\end{remark}

From now on, we will focus our efforts on describing the lengths $l_j=L>0$ for which the observability \eqref{Obs. Ineq. 2} holds. We strictly follow the argument used in \cite{Rosier}.
\begin{lemma}\label{lemma 3.4}
	Let $T>0$ be given. If the observability inequality \eqref{Obs. Ineq. 2} does not occur, then there exists $\varphi^T \in X_m$ with $\|\varphi^T\|_{X_m}=1$ for which the corresponding solution $\varphi$ of \eqref{Adj} satisfies
	\begin{align}\label{eq 3.6}
		\partial_{x}\varphi_{j}(\cdot,l_j)=0,\ j=1,\dots,m\text{ \ \ and \ \ }\partial_{x}^{2}\varphi_{j}(\cdot,l_j)=0,\ j=m+1,\dots,N.
	\end{align}
\end{lemma}
\begin{proof}
	Suppose that \eqref{Obs. Ineq. 2} is false. Then given $n \in \mathbb{N}$ there exists $\varphi^{T,n}\in X_m\backslash\{0\}$ such that
	\begin{align}\label{eq 3.7}		\sum_{j=1}^{m}\|\partial_{x}\varphi_{j}^n(\cdot,l_j)\|_{L^2(0,T)}^2+\sum_{j=m+1}^{n}\|\partial_{x}^{2}\varphi_{j}^n(\cdot,l_j)\|_{L^2(0,T)}^2<\frac{1}{n},\ \ \forall n \in \mathbb{N}.
	\end{align}	
	First note that thanks to Proposition \ref{W.P.L^2xH_0^1} we have for $j=1,\dots,N$
	\begin{align*}
		\|\varphi_j^n\|_{L^2(0,T,H^1(0,l_j))}\leq \|\varphi^n\|_{L^2(0,T;\mathbb{H}_e^1(\mathcal{T}))}\leq \|\varphi^n\|_{L^2(0,T;Y_m)}\leq C_9\|\varphi^{T,n}\|_{X_m}=C_9,
	\end{align*}
	that is, $(\varphi_j^n)_{n \in \mathbb{N}}$ is bounded in $L^2(0,T;H^1(0,l_j))$. On the other hand, using Lemma \ref{Ay in H^-2}
	\begin{align*}
		\|\partial_{t}\varphi_{j}(t,\cdot)\|_{H^{-2}(0,l_j)}=\|-\partial_{x}\varphi_{j}^n(t,\cdot)-\partial_{x}^{3}\varphi_{j}^n(t,\cdot)\|_{H^{-2}(0,l_j)}\leq \|\varphi_j^n(t,\cdot)\|_{H^1(0,l_j)}
	\end{align*}
	from where do we obtain
	\begin{align*}
		\|\partial_{t}\varphi_{j}^n\|_{L^2(0,T;H^{-2}(0,l_j))}\leq \|\varphi_j^n\|_{L^2(0,T;H^1(0,l_j))}\leq C.
	\end{align*}
	Thus $(\partial_{t}\varphi_{j}^n)_{n \in \mathbb{N}}$ is bounded in $L^2(0,T;H^{-2}(0,l_j))$. Since the first embedding in 
$
		H^1(0,l_j)\hookrightarrow L^2(0,l_j)\hookrightarrow H^{-2}(0,l_j)
$	is compact, from Aubin-Lions Lemma (see \cite{Aubin,Simon}) it follows that $(\varphi_j^n)_{n\in \mathbb{N}}$ is relatively compact in $L^2(0,T;L^2(0,l_j))$ and therefore it admits a convergent subsequence in $L^2(0,T;L^2(0,l_j))$.
	
	Analogously, the Proposition \ref{W.P.L^2xH_0^1} gives us for $j=m+1,\dots,N$
	\begin{align*}
		\|\varphi_j^n\|_{L^2(0,T;H^2(0,l_j))}\leq \|\varphi^n\|_{L^2(0,T;Y_m)}\leq C \|\varphi^{T,n}\|_{X_m},
	\end{align*}
	that is, $(\varphi_j^n)_{n \in \mathbb{N}}$ is bounded in $L^2(0,T;H^2(0,l_j))$. As $(\partial_{t}\varphi_{j}^n)_{n \in \mathbb{N}}$ is bounded in $L^2(0,T;H^{-2}(0,l_j))$ and the first embedding in 
$
		H^2(0,l_j)\hookrightarrow H^1(0,l_j)\hookrightarrow H^{-2}(0,l_j)
$	is compact, from Aubin-Lions Lemma (see \cite{Aubin,Simon}) we have $(\varphi_j^n)_{n \in \mathbb{N}}$ relatively compact in $L^2(0,T;H^1(0,l_j))$, that is, it has a convergent subsequence in $L^2(0,T;H^1(0,l_j))$.
	
	Now, due to Proposition \ref{trace-Adj.}, the sequence $(\varphi_1^n(\cdot,0))_{n\in \mathbb{N}}$ is bounded in $H^\frac{1}{3}(0,T)$ and thanks to compact embedding $H^\frac{1}{3}(0,T)\hookrightarrow L^2(0,T)$ we can extract from it a convergent subsequence in $L^2(0,T)$. In the same way, $(\varphi_j^{T,n})_{n \in \mathbb{N}}$ has a convergent subsequence in $L^2(0,l_j)$ due to compact embedding $H_0^1(0,l_j)\hookrightarrow L^2(0,l_j)$. In summary, there exists a subsequence $(\varphi^{T,n_k})_{k\in \mathbb{N}}$ of $(\varphi^{T,n})_{n\in \mathbb{N}}$ satisfying the following items:
	\begin{align}\label{subseq.}
		\begin{cases}
			(\varphi_j^{n_k})_{k\in \mathbb{N}} \text{ converges in }L^2(0,T;L^2(0,l_j)),\ j=1,\dots,N,\\
			(\varphi_j^{n_k})_{k\in \mathbb{N}} \text{ converges in } L^2(0,T;H^1(0,l_j)),\  j=m+1,\dots,N,\\
			(\varphi_1^{n_k}(\cdot,0))_{k\in \mathbb{N}} \text{ converges in } L^2(0,T),\\
			(\varphi_j^{T,n_k})_{k\in \mathbb{N}} \text{ converges in } L^2(0,l_j),\  j=m+1,\dots,N.
		\end{cases}
	\end{align}

Now, let $(\varphi^{T,n_k})_{k\in \mathbb{N}}$ the  subsequence of $(\varphi^{T,n})_{n\in \mathbb{N}}$ satisfying \eqref{subseq.}. From Proposition \ref{W.P.L^2xH_0^1} we have	
	\begin{align*}
		\|\varphi^{T,n_k}-\varphi^{T,n_r}\|_{X_m}^2\leq& \frac{1}{T}\sum_{j=1}^N\|\varphi_j^{n_k}-\varphi_j^{n_r}\|_{L^2(0,T;L^2(0,l_j)}^2+\left(\frac{1}{T}+\frac{6}{\sigma}\right)\sum_{j=m+1}^N\|\varphi_j^{n_k}-\varphi_j^{n_r}\|_{L^2(0,T;H^1(0,l_j)}^2\\
		&+2\left(\alpha-\frac{N}{2}\right)\|\varphi_1^{n_k}(\cdot,0)-\varphi_1^{n_r}(\cdot,0)\|_{L^2(0,T}^2+2\sum_{j=m+1}^N\|\varphi_j^{T,n_k}-\varphi_j^{T,n_r}\|_{L^2(0,l_j)}^2\\
		&+\sum_{j=1}^m\|\partial_{x}\varphi_{j}^{n_k}(\cdot,l_j)-\partial_{x}\varphi_{j}^{n_r}(\cdot,l_j)\|_{L^2(0,T)} ^2+\sum_{j=m+1}^N\|\partial_{x}^{2}\varphi_{j}^{n_k}(\cdot,l_j)-\partial_{x}^{2}\varphi_{j}^{n_r}(\cdot,l_j)\|_{L^2(0,T)}^2,
	\end{align*}
	which gives $(\varphi^{T,n_k})_{k\in \mathbb{N}}$ is a Cauchy sequence in $X_m$ so, there exists $\varphi^T\in X_m$ such that $\varphi^{T,n_k}\rightarrow \varphi^T$ in $X_m$. Consider $\varphi=S(T-\cdot)\varphi^T$. Since $X_m\hookrightarrow \mathbb{L}^2(\mathcal{T})$ continuously, from Propositions \ref{reg. Adj.} and \ref{phi_xx in L^2} ensures that
	\begin{align*}	
	\begin{cases}	
		\partial_{x}\varphi_{j}^{n_k}(\cdot,l_j)\rightarrow\partial_{x}\varphi_{j}(\cdot,l_j)\text{ \ in \ }L^2(0,T),\ \  j=1,\dots,m,\\		
		\partial_{x}^{2}\varphi_{j}^{n_k}(\cdot,l_j)\rightarrow\partial_{x}^{2}\varphi_{j}(\cdot,l_j)\text{ \ in \ }L^2(0,T),\ \  j=m+1,\dots,N.
		\end{cases}
	\end{align*}
	From \eqref{eq 3.7} we get
	\begin{align*}
		\partial_{x}\varphi_{j}(\cdot,l_j)=0,\ \  j=1,\dots,m
		\text{ \ \ and \ \ }
		\partial_{x}^{2}\varphi_{j}(\cdot,l_2)=0,\ \  j=m+1,\dots,N.
	\end{align*}
Since $\|\varphi^{T,n}\|_{X_m}=1$ for every $n \in \mathbb{N}$, we also have, in the limit, that $\|\varphi^T\|_{X_m}=1$. So, Lemma \ref{lemma 3.4} holds.
\end{proof}
From now on, we deal with the case $\alpha=N$.
\begin{lemma}\label{lemma 3.5}
	Let $T>0$ and denote by $N_T$ the space of data $\varphi^T \in X_m$ whose respective solutions $\varphi=S(T-\cdot)\varphi^T$ of \eqref{Adj} satisfy \eqref{eq 3.6}. If $N_T\ne \{0\}$ then there exists $\lambda \in \mathbb{C}$ and $\varphi\in \mathbb{H}^3(\mathcal{T})\backslash\{0\}$ such that
	\begin{align}\label{spectral problem}
		\begin{cases}			\lambda\varphi_j+\varphi_{j}^{\prime}+\varphi_{j}^{\prime\prime\prime}=0,&x \in (0,l_j),\  j=1,\dots,N,\\
			\varphi_j(0)=\varphi_1(0),& j=1,\dots,N,\\
			\displaystyle\sum_{j=1}^N\varphi_{j}^{\prime\prime}(0)=0,&\\		
			\varphi_j(l_j)=\varphi_{j}^{\prime}(0)=0,& j=1,\dots,N,\\
			\varphi_{j}^{\prime}(l_j)=0,& j=1,\dots,m,\\
			\varphi_{j}^{\prime\prime}(l_j)=0,& j=m+1,\dots,N.
		\end{cases}
	\end{align}
\end{lemma}
\begin{proof} Using the arguments as those given in \cite[Lemma 3.4]{Rosier}, follows that if $N_T \neq \emptyset$, the map $\varphi^T \in N_T \mapsto \tilde{A}\left(N_T\right) \subset \mathbb{C} N_T$ (where $\mathbb{C} N_T$ denote the complexification of $N_T$ ) has (at least) one eigenvalue; hence, there exists $\lambda \in \mathbb{C}$ and $\varphi^T\in N_T\backslash\{0\}\subset \mathcal{D}(A^*)\subset \mathbb{H}^3(\mathcal{T})$ such that \eqref{spectral problem} holds, which give us the lemma. 
\end{proof}

\subsection{Entire function related to each condition}
Each observability is linked to a corresponding entire function, which plays a crucial role in the analysis. From this point forward, we will organize our analysis into the following  six cases:
\begin{align}\label{cases}
\begin{aligned}
	&\bullet\ N=2 \text{ and }m=1;\\
	&\bullet\ N\geq 3 \text{ and } m=1;\\
	&\bullet\ N\geq 3 \text{ and } m=N-1;\\
	&\bullet\ N>3 \text{ and } 1<m<N-1;\\
	&\bullet N\geq 2\text{ and }m=N\text{ (Full Neumann)}; \\
	&\bullet N\geq 2\text{ and }m=0\ \text{ (Full Dirichlet)}, \\
\end{aligned}
\end{align}
remembering that $l_j=L$, for all $j$.

\subsubsection{\textbf{Case \texorpdfstring{$N=2$}{N=2} and $m=1$}} In this case, we consider a slightly more general problem, namely 
\begin{align}\label{spectral problem-2}
	\begin{cases}
		\lambda\varphi_j+\varphi_{j}^{\prime}+\varphi_{j}^{\prime\prime\prime}=0,&x \in (0,l_j),\ j=1,2,\\
		\varphi_2(0)=\varphi_1(0),&\\
		\varsigma_{Ne}\varphi_{1}^{\prime\prime}(0)+\varsigma_{Di}\varphi_{2}^{\prime\prime}(0)=0,&\\		
		\varphi_j(l_j)=\varphi_{j}^{\prime}(0)=0,&j=1,2,\\
		\varphi_{1}^{\prime}(l_1)=\varphi_{2}^{\prime\prime}(l_2)=0,&
	\end{cases}
\end{align}
where $\varsigma_{Ne}$, $\varsigma_{Di}>0$. Problem \eqref{spectral problem} corresponds to the case $\varsigma_{Ne}=\varsigma_{Di}=1$. We can prove that the observability inequality \eqref{Obs. Ineq. 2} holds for any $L>0$. This is a consequence of the following lemma.
\begin{lemma}\label{lemma 3.6}
	Let $L>0$ and suppose $l_1=l_2=L$. For any $\lambda\in \mathbb{C}$, does not exist $\varphi\in \mathbb{H}^3(\mathcal{T})\backslash\{0\}$ satisfying \eqref{spectral problem-2}.
\end{lemma}
\begin{proof}
	For $\psi\in \mathbb{L}^2(\mathcal{T})$, we introduce the notation $$\hat{\psi_j}(\xi)=\displaystyle\int_0^{l_j}\psi_j(x)e^{-ix\xi}dx.$$ Consider $\lambda\in \mathbb{C}$ and suppose that $\varphi\in \mathbb{H}^3(\mathcal{T})$ satisfies \eqref{spectral problem-2}. We will show that the unique solution for \eqref{spectral problem-2} is the trivial one, that is, $\varphi=0$.
	
	Multiplying the system \eqref{spectral problem-2} by $e^{-ix\xi}$, integrating by parts in $(0,l_j)$ and using the boundary conditions we get, for every $\xi \in \mathbb{C}$,
	\begin{align*}
		\left[(i\xi)^3+i\xi+\lambda\right]&=\left[(1-\xi^2)\varphi_j(0)-i\xi e^{-il_j\xi}\varphi_{j}^{\prime}(l_j)-e^{-il_j\xi}\varphi_{j}^{\prime\prime}(l_j)+\varphi_{j}^{\prime\prime}(0)\right],\ j=1,2.
	\end{align*}
	Writing $\lambda =ip$ with $p \in \mathbb{C}$ and multiplying this equation by $i$ yields
	\begin{align*}
		\left(\xi^3-\xi-p\right)\hat{\varphi_j}(\xi)&=i\left[(1-\xi^2)\varphi_j(0)-i\xi e^{-il_j\xi}\varphi_{j}^{\prime}(l_j)-e^{-il_j\xi}\varphi_{j}^{\prime\prime}(l_j)+\varphi_{j}^{\prime\prime}(0)\right],\ j=1,2.
	\end{align*}
	Setting
	\begin{align*}
		\kappa=\varphi_1(0)=\varphi_j(0),\ \ \ \delta_j=-\varphi_{j}^{\prime}(l_j),\ \ \ \gamma_j=-\varphi_{j}^{\prime\prime}(l_j),\ \ \ 
		\beta_j=\varphi_{j}^{\prime\prime}(0)
	\end{align*}
	one can write
	\begin{align*}
		\left(\xi^3-\xi-p\right)\hat{\varphi_j}(\xi)&=i\left[(1-\xi^2)\kappa+i\delta_j\xi e^{-il_j\xi}+\gamma_j e^{-il_j\xi}+\beta_j\right],\ \ j=1,2.
	\end{align*}
	Since $\delta_1=\gamma_2=0$ and $l_1=l_2=L$ we have
	\begin{align}
		&\left(\xi^3-\xi-p\right)\hat{\varphi}_1(\xi)=i\left[(1-\xi^2)\kappa+\gamma_1 e^{-iL\xi}+\beta_1\right],\ \forall\ \xi \in \mathbb{C},\label{2.2}\\
		&\left(\xi^3-\xi-p\right)\hat{\varphi}_2(\xi)=i\left[(1-\xi^2)\kappa+i\delta_2\xi e^{-iL\xi}+\beta_2\right],\ \ \forall\ \xi \in  \mathbb{C}.\label{2.3}
	\end{align}
Now, defining $f=\varphi_1-\varphi_2$ and $\beta=\beta_1-\beta_2$, from the above identities, we obtain
	\begin{align}\label{2.4}
		\left(\xi^3-\xi-p\right)\hat{f}(\xi)=i\left[(\gamma_1-i\delta_2\xi)e^{-iL\xi}+\beta\right],\ \forall\ \xi \in \mathbb{C}.
	\end{align}
	\begin{claim}\label{claim 3.1}
		If $\gamma_1=0$ or $\beta=0$ then $\varphi=0$.
	\end{claim}
	\begin{proof}
		If $\gamma_1=0$ then $\varphi_1$ solves the problem
		\begin{align*}
			\begin{cases}			\lambda\varphi_1+\varphi^{\prime}_{1}+\varphi^{\prime\prime\prime}_{1}=0,\\
				\varphi_1(L)=\varphi^{\prime}_{1}(L)=\varphi^{\prime\prime}_{1}(L)=0,
			\end{cases}
		\end{align*}
		so $\varphi_1=0$. Consequently,$\varphi_2$ satisfies
		\begin{align*}
			\begin{cases}			\lambda\varphi_2+\varphi^{\prime}_{2}+\varphi^{\prime\prime\prime}_{2}=0,\\
				\varphi_2(0)=\varphi^{\prime}_{2}(0)=\varphi^{\prime\prime}_{2}(0)=0,
			\end{cases}
		\end{align*}
		then $\varphi_2=0$ and therefore $\varphi=0$.
		
		Now, if $\beta=0$ then $\beta_1=\beta_2$ so $\varphi_1$ and $\varphi_2$ are solutions to the problem
		\begin{align*}
			\begin{cases}		
				\lambda \varphi_i+\varphi^{\prime}_{i}+\varphi^{\prime\prime\prime}_{i}=0,\quad i=1,2\\
				\varphi_i(0)=\kappa,\quad \varphi_i^{\prime}(0)=0, \quad 	\varphi_i^{\prime\prime}(0)=\beta_1.
			\end{cases}
		\end{align*}
		By uniqueness of solution, it follows that $\varphi_1=\varphi_2$, which implies $f=0$. Hence $\hat{f}=0$ and since $\beta=0$, \eqref{2.4} becomes
		\begin{align*}
			(\gamma_1-i\delta_2\xi)e^{-iL\xi}=0,\ \forall \xi \in \mathbb{C}.
		\end{align*}
		Evaluating this equality at $\xi=0$, we obtain $\gamma_1=0$, and, as seen before, this leads us to $\varphi=0$, showing Claim \ref{claim 3.1}.
	\end{proof}
	
	Assuming 
	\begin{align}\label{2.5}
		\gamma_1\ne 0\text{ \ and \ }\beta\ne 0,
	\end{align}
	with the following claim in hand, the Lemma \ref{lemma 3.6}  holds for the case $\lambda=0$.
	\begin{claim}\label{claim 2.2}
		If $\lambda=0$ then does not exist $\varphi=(\varphi_1,\varphi_2)\in \mathbb{H}^3(\mathcal{T})\backslash\{0\}$ satisfying \eqref{spectral problem-2}.
	\end{claim}
	\begin{proof}
		We will show that for $\lambda=0$, then we necessarily have $\varphi=0$. To do this, let us multiply \eqref{2.2} and \eqref{2.3} by $-i$ to get
		\begin{align}
			&-i\left(\xi^3-\xi\right)\hat{\varphi}_1(\xi)=(1-\xi^2)\kappa+\gamma_1 e^{-iL\xi}+\beta_1,\ \forall\ \xi \in \mathbb{C}\label{2.6},\\
			&-i\left(\xi^3-\xi\right)\hat{\varphi}_2(\xi)=(1-\xi^2)\kappa+i\delta_2\xi e^{-iL\xi}+\beta_2,\ \ \forall\ \xi \in  \mathbb{C}.\label{2.7}
		\end{align}
		Evaluating \eqref{2.6} at $\xi=1$ and $\xi=-1$ we obtain, respectively,
		\begin{align*}
			\gamma_1e^{-iL}=-\beta_1\text{ \ and \ }\gamma_1e^{iL}=-\beta_1
		\end{align*}
		from where we have $\gamma_1e^{-iL}=\gamma_1e^{iL}$ and, since $\gamma_1\ne 0$, it follows that
		\begin{align}\label{2.8}
			e^{-iL}=e^{iL}.
		\end{align}
		On the other hand, evaluating \eqref{2.7} at $\xi=1$ and $\xi=-1$ we obtain, respectively,
		\begin{align*}
			i\delta_2e^{-iL}=-\beta_2\text{ \ and \ }-i\delta_2e^{iL}=-\beta_2
		\end{align*}
		which gives us
		$$i\delta_2e^{-iL}=-i\delta_2e^{iL}.$$
		If $\delta_2\ne 0$, then the above equality provides
		\begin{align}\label{2.9}
			e^{-iL}=-e^{iL}.
		\end{align}
		Adding \eqref{2.8} and \eqref{2.9} we get $2e^{-iL}=0$, which is not possible. Therefore, $\delta_2=0$ so $\varphi_2$ solves the problem
		\begin{align*}
			\begin{cases}			\lambda\varphi_2+\varphi^{\prime}_{2}+\varphi^{\prime\prime\prime}_{2}=0,\\
				\varphi_2(L)=\varphi^{\prime}_{2}(L)=\varphi^{\prime\prime}_{2}(L)=0,
			\end{cases}
		\end{align*}
		and consequently $\varphi_2=0$. Thanks to this fact and using the boundary conditions of \eqref{spectral problem-2}, $\varphi_1$ is the solution of the following system 
		\begin{align*}
			\begin{cases}			\lambda\varphi_1+\varphi^{\prime}_{1}+\varphi^{\prime\prime\prime}_{1}=0,\\
				\varphi_1(0)=\varphi^{\prime}_{1}(0)=\varphi^{\prime\prime}_{1}(0)=0,
			\end{cases}
		\end{align*}
		which tells us that $\varphi_1=0$ and, consequently $\varphi=0$, which give the Claim \ref{claim 2.2}.
	\end{proof}
	
	Define $P_\lambda(\mu)=\mu^3+\mu+\lambda$. Before studying the case $\lambda\ne 0$, we need to establish an important relationship between the multiplicity of the roots of $P_\lambda$ and the solutions of \eqref{spectral problem-2}.
	\begin{claim}\label{claim 2.3} If $P_\lambda$ admits multiples roots, then $\varphi=0$.
	\end{claim}
	\begin{proof}
		Let $\mu_0, \mu_1$ and $\mu_2$ the roots of $P_\lambda$. By Girard's relations, we have
		\begin{align}
			\mu_0+\mu_1+\mu_2&=0,\label{2.10}\\
			\mu_0\mu_1+\mu_1\mu_2+\mu_0\mu_2&=1,\label{2.11}\\
			\mu_0\mu_1\mu_2&=-\lambda\label{2.12}.
		\end{align}
		If $P_\lambda$ has a triple root thus, by \eqref{2.10}, we have
		$$\mu_0=\mu_1=\mu_2=0.$$
So, from \eqref{2.12} it follows that $\lambda=0$. With this, the Claim \ref{claim 2.2} ensures that $\varphi=0$.
		
		Now, suppose that $P_\lambda$ has a double root, let us say $\mu_1=\mu_2=\sigma$. Relation \eqref{2.10} gives $\mu_0=-2\sigma$ and substituting it in \eqref{2.11} 
		leads us to
		\begin{align}\label{2.13}
			\sigma=\pm\frac{i}{\sqrt{3}}.
		\end{align}
		By the theory of ordinary differential equations, the set
		\begin{align*}
			\left\{e^{-2\sigma x},e^{\sigma x}, xe^{\sigma x}\right\}
		\end{align*}
		is a generator of the solutions of the following ordinary differential equation
		\begin{align*}
			\lambda y+y^{\prime}+y^{\prime\prime\prime}=0.
		\end{align*}

Note that $\varphi_2$ solution of
		\begin{align*}
			\begin{cases}			\lambda\varphi_2+\varphi^{\prime}_{2}+\varphi^{\prime\prime\prime}_{2}=0,\\
				\varphi_2(L)=\varphi^{\prime}_{2}(0)=\varphi^{\prime\prime}_{2}(L)=0,
			\end{cases}
		\end{align*}
ensures the existence of  $d_0,d_1,d_2 \in \mathbb{C}$ such that
		\begin{align*}
			\varphi_2(x)=d_0e^{-2\sigma x}+d_1e^{\sigma x}+d_2xe^{\sigma x}.
		\end{align*}
		Consequently
		\begin{equation}\label{2.14a}
				\varphi^{\prime}_{2}(x)=d_0(-2\sigma)e^{-2\sigma x}+d_1\sigma e^{\sigma x}+d_2(e^{\sigma x}+x\sigma e^{\sigma x}),
\end{equation}
and
		\begin{equation}\label{2.14b}
				\varphi^{\prime\prime}_{2}(x)=d_0(-2\sigma)^2e^{-2\sigma x}+d_1\sigma^2 e^{\sigma x}+d_2(\sigma e^{\sigma x}+\sigma e^{\sigma x}+x\sigma^2 e^{\sigma x}).
			\end{equation}
Observe that the boundary condition $\varphi_2(L)=0$ gives us
		\begin{align*}
			d_1+Ld_2=-d_0e^{-3\sigma L}.
		\end{align*}
		On the other hand, the condition $\varphi^{\prime\prime}_{2}(L)=0$ provides, due the relation given by \eqref{2.14b}, that
		\begin{align*}
			d_1\sigma+d_2(2+L\sigma)=-4d_0\sigma e^{-3\sigma L}.
		\end{align*}
		Thus, we have the system
		\begin{align*}
			\left\{
			\begin{array}{l}
				d_1+Ld_2=-d_0e^{-3\sigma L}\\
				d_1\sigma+d_2(2+L\sigma)=-4d_0\sigma e^{-3\sigma L}
			\end{array}
			\right.
		\end{align*}
		Below, we solve this system for $d_1$ and $d_2$ using Gauss-Jordan elimination (we denote by $r_k$ the $k$-th row of a matrix):
		\begin{align*}
			&\left[
			\begin{array}{ccc}
				1&L&-d_0e^{-3\sigma L}\\
				\sigma&2+L\sigma&-4d_0\sigma e^{-3\sigma L}
			\end{array}
			\right]\sim
			\left[
			\begin{array}{ccc}
				1&0&-d_0e^{-3\sigma L}+\frac{3d_0\sigma Le^{-3\sigma L}}{2}\\
				0&1&-\frac{3d_0\sigma e^{-3\sigma L}}{2}
			\end{array}
			\right].
		\end{align*}
		Therefore 
		\begin{align}\label{2.15}
			d_1=-d_0e^{-3\sigma L}+\frac{3d_0\sigma Le^{-3\sigma L}}{2}\text{ \ \ and \ \ }d_2=-\frac{3d_0\sigma e^{-3\sigma L}}{2}.
		\end{align}
		
		Using this values in \eqref{2.14a} yields,
		\begin{align*}
			\varphi_{2x}(x)&=d_0(-2\sigma)e^{-2\sigma x}+\left(-d_0e^{-3\sigma L}+\frac{3d_0\sigma Le^{-3\sigma L}}{2}\right)\sigma e^{\sigma x}+\left(-\frac{3d_0\sigma e^{-3\sigma L}}{2}\right)(e^{\sigma x}+x\sigma e^{\sigma x}).
		\end{align*}
		Since $\varphi_{2x}(0)=0$ we get that
		\begin{align*}
			-2d_0\sigma-d_0\sigma e^{-3\sigma L}+\frac{3d_0\sigma^2 Le^{-3\sigma L}}{2}-\frac{3d_0\sigma e^{-3\sigma L}}{2}=0.
		\end{align*}
		Multiplying the previous equality by $2/\sigma$, we have
		\begin{align*}
			d_0\left(-4-2e^{-3\sigma L}+3\sigma Le^{-3\sigma L}-3 e^{-3\sigma L}\right)=0.
		\end{align*}
		If $d_0\ne 0$, then it follows that
		\begin{align*}
			-4-2e^{-3\sigma L}+3\sigma Le^{-3\sigma L}-3 e^{-3\sigma L}=0 \iff \left(-5+3\sigma L\right)e^{-3\sigma L}=4.
		\end{align*}
		Due to \eqref{2.13} this equality becomes
		\begin{align*}
			\left(-5\pm \frac{3iL}{\sqrt{3}}\right)e^{\pm \frac{3iL}{\sqrt{3}}}=4.
		\end{align*}
		Taking the absolute value in the previous equality, we see that
		\begin{align*}
			\left|-5\pm \frac{3iL}{\sqrt{3}}\right|=4 \implies 25+3L^2=16 \implies 3L^2<0,
		\end{align*}
a contradiction. Therefore $d_0=0$ and from \eqref{2.15} it follows that $d_1=d_2=0$, which results in $\varphi_2=0$. Hence, using the boundary conditions of \eqref{spectral problem-2}, $\varphi_1$ solves the problem
		\begin{align*}
			\begin{cases}			\lambda\varphi_1+\varphi^{\prime}_{1}+\varphi^{\prime\prime\prime}_{1}=0,\\
				\varphi_1(0)=\varphi^{\prime}_{1}(0)=\varphi^{\prime\prime}_{1}(0)=0,
			\end{cases}
		\end{align*}
		that is, $\varphi_1=0$ and therefore $\varphi=0$, giving the Claim \ref{claim 2.3}.
	\end{proof}
	
	According to claims \ref{claim 2.2} and \ref{claim 2.3}, it remains to study the case where $\lambda\in \mathbb{C}\backslash\{0\}$ and the roots of $P_\lambda$ are all simple. In this case, defining $P(\xi)=\xi^3-\xi-p$, due to identity
	\begin{align*}
		P_\lambda(i\xi)=-iP(\xi),\ \forall \xi \in \mathbb{C},
	\end{align*}
	we have that the roots of $P$ are all simple, too. Moreover, as $\lambda=ip$ we must have
$$
		p \in \mathbb{C}\backslash\{0\}.
$$
	With this in mind, the next claim completes the proof of Lemma \ref{lemma 3.6}.
	\begin{claim}\label{claim 3.4}
		If $\lambda \in \mathbb{C}\backslash\{0\}$ and the roots of $P_\lambda$ are all simple then $\delta_2=0$.
	\end{claim}
	\begin{proof}
		Suppose, by contradiction, that $\delta_2\ne 0$. Let $\xi_0,\xi_1,\xi_2$ the roots of $P$. As $p\ne 0$ we have $\xi_s\ne 0$ for $s=0,1,2$. The relation \eqref{2.2} ensures that
		\begin{align*}
			(1-\xi_s^2)\kappa+\gamma_1 e^{-iL\xi_s}+\beta_1=0.
		\end{align*}
		Multiplying this equation by $1/\gamma_1$ yields
		\begin{align}\label{2.17}
			(1-\xi_s^2)\kappa\frac{1}{\gamma_1}+e^{-iL\xi_s}+\frac{\beta_1}{\gamma_1}=0,\ s=0,1,2.
		\end{align}
		Now, from \eqref{2.3}, note that
		\begin{align*}
			(1-\xi_s^2)\kappa+i\delta_2\xi_s e^{-iL\xi_s}+\beta_2=0.
		\end{align*}
		Multiplying this equation by $1/i\delta_2\xi_s$, we obtain
		\begin{align}\label{2.18}
			(1-\xi_s^2)\kappa\frac{1}{i\delta_2\xi_s}+e^{-iL\xi_s}+\frac{\beta_2}{i\delta_2\xi_s}=0,\ s=0,1,2.
		\end{align}
		Taking the difference between \eqref{2.18} and \eqref{2.17} and multiplying the result by  $\gamma_1 i\delta_2\xi_s$  yields that
		\begin{align*}
			-\kappa i\delta_2\xi_s^3+\kappa\gamma_1\xi_s^2+(\kappa i\delta_2+\beta_1 i\delta_2)\xi_s-(\kappa\gamma_1+\beta_2\gamma_1)=0,\ s=0,1,2.
		\end{align*}
		Since $\xi_0,\xi_1,\xi_2$ are all distinct, we conclude that $\xi_s$ is a simple root of the polynomial $\tilde{P}$ defined by
		\begin{align*}
			\tilde{P}(\xi)=-\kappa i\delta_2\xi^3+\kappa\gamma_1\xi^2+(\kappa i\delta_2+\beta_1i\delta_2)\xi-(\kappa\gamma_1+\beta_2\gamma_1),
		\end{align*}
		 for $s=0,1,2$. Thus, there exist $c \in \mathbb{C}\backslash\{0\}$ such that
		\begin{align*}
			P(\xi)=c\tilde{P}(\xi),\ \forall\xi \in\mathbb{C}.
		\end{align*}
		Hence, the corresponding coefficients of $P$ and $c\tilde{P}$ must be equal. In particular,
		\begin{align*}
			-c\kappa i\delta_2=1\text{ \ \ and \ \ }c\kappa\gamma_1=0.
		\end{align*}
The first equality tells us that $\kappa\ne 0$, since in our hypothesis $\delta_2\neq0$. Thus, using this fact in the second equality, we obtain $\gamma_1=0$, which is a contradiction with \eqref{2.5}; consequently, Claim \ref{claim 3.4} is complete.
	\end{proof}
	So, the previous claims ensure the Lemma \ref{lemma 3.6}.
\end{proof}
\subsubsection{\textbf{Case \texorpdfstring{$N\geq3$ and $m=1$}{N>=3 and m=1}}} In this case the problem \eqref{spectral problem} becomes
\begin{align}\label{spectral problem-3}
	\begin{cases}
		\lambda\varphi_j+\varphi^{\prime}_{j}+\varphi^{\prime\prime\prime}_{j}=0,&x \in (0,l_j),\ j=1,...,N,\\
		\varphi_j(0)=\varphi_1(0),&j=1,...,N,\\
		\displaystyle\sum_{j=1}^N\varphi^{\prime\prime}_{j}(0)=0,&\\		
		\varphi_j(l_j)=\varphi^{\prime}_{j}(0)=0,&j=1,...,N,\\
		\varphi^{\prime}_{1}(l_1)=0,\ \ \varphi_{j}^{\prime\prime}(l_j)=0,&j=2,...,N,
	\end{cases}
\end{align}
and we have the following result.
\begin{lemma}\label{lemma 3.6-b}
	Let $L>0$ and assume $l_j=L$, for any $j=1,...,N$. There exist $\lambda\in \mathbb{C}$ and $\varphi\in \mathbb{H}^3(\mathcal{T})\backslash\{0\}$ satisfying \eqref{spectral problem-3} if and only if $L \in \mathcal{N}^*$.
\end{lemma}
\begin{proof}
If $L \in \mathcal{N}^*$ then, from \cite[Propositions 1 and 2]{GG1}, there exist $\lambda \in\mathbb{C}$ and $z \in H^3(0,L)\backslash\{0\}$ such that
	\begin{align*}
		\begin{cases}
			\lambda z+z^{\prime}+z^{\prime\prime\prime}=0,\\
			z(0)=z(L)=z^{\prime}(0)=z^{\prime\prime}(L)=0.
		\end{cases}
	\end{align*}
	Defining $\varphi_1=0,\ \varphi_2=z,\ \varphi_3=-z,\ \varphi_j=0\text{ for }j=4,...,N\text{ and }\varphi=(\varphi_1,\varphi_2,...,\varphi_N)$ then $\varphi \in \mathbb{H}^3(\mathcal{T})\backslash\{0\}$ and $(\lambda,\varphi)$ satisfies \eqref{spectral problem-3}.\\

Now, assuming that there exist $\lambda \in \mathbb{C}$ and $\varphi \in \mathbb{H}^3(\mathcal{T})\backslash\{0\}$ satisfying \eqref{spectral problem-3}, we will show by contradiction that this leads to $L \in \mathcal{N}^*$. Suppose that this does not occur, that is, $L \notin \mathcal{N}^*$. Define the following functions
	\begin{align*}
		\psi=\sum_{j=2}^N\varphi_j\text{ \ \ and \ \ }\psi_j=\psi-(N-1)\varphi_j,\ j=2,...,N.
	\end{align*}
	Note that, for each $j \in \{2,...,N\}$, $\psi_j$ solves the problem
	\begin{align}
    \label{eq: critical dir}
		\begin{cases}
			\lambda \psi_j+\psi^{\prime}_{j}+\psi^{\prime\prime}_{j}=0,\\
			\psi_j(0)=\psi_j(L)=\psi^{\prime}_{j}(0)=\psi^{\prime\prime}_{j}(L)=0.
		\end{cases}
	\end{align}
	Since $L\notin \mathcal{N}^*$, from \cite[Propositions 1 and 2]{GG1} it follows that
	\begin{align*}
		\psi_j=0,\text{for}\ j=2,...,N, \implies \varphi_j=\frac{1}{N-1}\psi,\text{for}\  j=2,...,N.
	\end{align*}
	Consequently,
	$$
		\varphi_j=\varphi_2,\ \text{for}\ j=2,...,N.
$$
	This implies that, $\varphi_{1}$ and $\varphi_{N}$ satisfy the spectral problem
    \begin{align*}
	\begin{cases}
		\lambda\varphi_j+\varphi^{\prime}_{j}+\varphi^{\prime\prime\prime}_{j}=0,&x \in (0,l_j),\ j=1,N,\\
		\varphi_N(0)=\varphi_1(0),&\\
		\varphi^{\prime\prime}_{1}(0)+(N-1)\varphi^{\prime\prime}_{N}(0)=0,&\\		
		\varphi_j(L)=\varphi^{\prime}_{j}(0)=0,&j=1,N,\\
		\varphi^{\prime}_{1}(L)=\varphi^{\prime\prime}_{N}(L)=0,&
	\end{cases}
\end{align*}
which corresponds to \eqref{spectral problem-2} in the case $\varsigma_{Ne}=1$ and $\varsigma_{Di}=N-1$. Therefore, the unique solution is $\varphi_{1}=\varphi_{N}=0$. Finally, as $0=\varphi_{N}=\varphi_{j}$ for $j=2,\dots,N$, we obtain $\varphi=0$.\end{proof}
\subsubsection{\textbf{Case \texorpdfstring{$N\geq3$ and $m=N-1$}{N>=3 and m=N-1}}}
In this case, the problem \eqref{spectral problem} becomes
\begin{align}\label{spectral problem-4}
\begin{cases}
\lambda\varphi_j+\varphi^{\prime}_{j}+\varphi^{\prime\prime\prime}_{j}=0,&x \in (0,l_j),\ j=1,...,N,\\
\varphi_j(0)=\varphi_1(0),&j=1,...,N,\\
\displaystyle\sum_{j=1}^N\varphi^{\prime\prime}_{j}(0)=0,&\\		
\varphi_j(l_j)=\varphi^{\prime}_{j}(0)=0,&j=1,...,N,\\
\varphi_{j}^{\prime}(l_j)=0,&j=1,...,N-1\\
\varphi^{\prime\prime}_{N}(l_N)=0,&
\end{cases}
\end{align}
and we have the following result.
\begin{lemma}\label{lemma 3.7}
	Let $L>0$ and assume $l_j=L$, for $j=1,...,N$. There exist $\lambda\in \mathbb{C}$ and $\varphi\in \mathbb{H}^3(\mathcal{T})\backslash\{0\}$ satisfying \eqref{spectral problem-4} if and only if $L \in \mathcal{N}$.
\end{lemma}
\begin{proof} If $L \in \mathcal{N}$ the result follows direct by \cite[Lemma 3.5]{Rosier} there exist $\lambda \in\mathbb{C}$ and $z \in H^3(0,L)\backslash\{0\}$ such that
	\begin{align*}
		\begin{cases}
			\lambda z+z^{\prime}+z^{\prime\prime\prime}=0,\\
			z(0)=z(L)=z^{\prime}(0)=z^{\prime}(L)=0.
		\end{cases}
	\end{align*}
	Defining $\varphi_1=z,\ \varphi_2=-z,\ \varphi_j=0\text{ for }j=3,...,N\text{ and }\varphi=(\varphi_1,\varphi_2,...,\varphi_N)$ then $\varphi \in \mathbb{H}^3(\mathcal{T})\backslash\{0\}$ with $(\lambda,\varphi)$ satisfying \eqref{spectral problem-4}.
	
	Conversely, assume that there exist $\lambda \in \mathbb{C}$ and $\varphi \in \mathbb{H}^3(\mathcal{T})\backslash\{0\}$ satisfying \eqref{spectral problem-4}. We will show that $L \in \mathcal{N}$. Suppose, by contradiction, that $L \notin \mathcal{N}$ and define
	\begin{align*}
		\psi=\sum_{j=1}^{N-1}\varphi_j\text{ \ \ and \ \ }\psi_j=\psi-(N-1)\varphi_j,\ j=1,...,N-1.
	\end{align*}
	For each $j \in \{1,...,N-1\}$, $\psi_j$ solves the problem
	\begin{align*}
		\begin{cases}
			\lambda \psi_j+\psi^{\prime}_{j}+\psi^{\prime\prime\prime}_{j}=0,\\
						\psi_j(0)=\psi_j(L)=\psi^{\prime}_{j}(0)=\psi^{\prime}_{j}(L)=0.
		\end{cases}
	\end{align*}
	Since $L\notin \mathcal{N}$, from \cite[Lemma 3.5]{Rosier} it follows that
	\begin{align*}
		\psi_j=0,\ j=1,...,N-1 \implies \varphi_j=\frac{1}{N-1}\psi,\ j=1,...,N-1.
	\end{align*}
	Consequently,
	$$
		\varphi_j=\varphi_1,\ j=1,...,N-1.
$$
	This implies that, $\varphi_{1}$ and $\varphi_{N}$ satisfy the spectral problem
    \begin{align*}
	\begin{cases}
		\lambda\varphi_j+\varphi^{\prime}_{j}+\varphi^{\prime\prime\prime}_{j}=0,&x \in (0,l_j),\ j=1,N,\\
		\varphi_N(0)=\varphi_1(0),&\\
		(N-1)\varphi^{\prime\prime}_{1}(0)+1\varphi^{\prime\prime}_{N}(0)=0,&\\		
		\varphi_j(L)=\varphi^{\prime}_{j}(0)=0,&j=1,N,\\
		\varphi^{\prime}_{1}(L)=\varphi^{\prime\prime}_{N}(L)=0,&
	\end{cases}
\end{align*}
which corresponds to \eqref{spectral problem-2} in the case $\varsigma_{Ne}=N-1$ and $\varsigma_{Di}=1$. Therefore, the unique solution is $\varphi_{1}=\varphi_{N}=0$. Finally, as $0=\varphi_{1}=\varphi_{j}$ for $j=1,\dots,N-1$, we obtain $\varphi=0$.\end{proof}

\subsubsection{\textbf{Case \texorpdfstring{$N>3$ and $1<m<N-1$}{N>3 and 1<m<N-1}}} We recall the problem \eqref{spectral problem} given by
\begin{align}\label{spectral problem-5}
\begin{cases}			
\lambda\varphi_j+\varphi^{\prime}_{j}+\varphi^{\prime\prime}_{j}=0,&x \in (0,l_j),\ j=1,...,N,\\
\varphi_j(0)=\varphi_1(0),&j=1,...,N,\\
\displaystyle\sum_{j=1}^N\varphi^{\prime\prime}_{j}(0)=0,&\\		
\varphi_j(l_j)=\varphi^{\prime\prime}_{j}(0)=0,&j=1,...,N,\\
\varphi^{\prime}_{j}(l_j)=0,&j=1,...,m,\\
\varphi^{\prime\prime}_{j}(l_j)=0,&j=m+1,...,N.
\end{cases}
\end{align}
For this case, we have the following result.
\begin{lemma}\label{lemma 3.8}
	Let $L>0$ and assume $l_j=L$ for $j=1,...,N$. There exist $\lambda\in \mathbb{C}$ and $\varphi\in \mathbb{H}^3(\mathcal{T})\backslash\{0\}$ satisfying \eqref{spectral problem-5} if and only if $L \in \mathcal{N}\cup \mathcal{N}^*$.
\end{lemma}
\begin{proof} If $L \in \mathcal{N}\cup \mathcal{N}^*$ the result is a direct consequence of \cite[Lemma 3.5]{Rosier} and \cite[Propositions 1 and 2]{GG1} (see Lemmas \ref{lemma 3.7} and \ref{lemma 3.6}).  
%
Conversely, suppose there exist \(\lambda \in \mathbb{C}\) and \(\varphi \in \mathbb{H}^3(\mathcal{T}) \setminus \{0\}\) satisfying \eqref{spectral problem-5}. We will demonstrate, by contradiction, that this implies \(L \in \mathcal{N} \cup \mathcal{N}^*\). Assume, for the sake of contradiction, that \(L \notin \mathcal{N} \cup \mathcal{N}^*\), and define
	\begin{align*}
		\psi=\sum_{j=1}^{m}\varphi_j\text{ \ \ and \ \ }\psi_j=\psi-m\varphi_j,\ j=1,...,m.
	\end{align*}
	For each $j \in \{1,...,m\}$, $\psi_j$ solves the problem
	\begin{align*}
		\begin{cases}
			\lambda \psi_j+\psi^{\prime}_{j}+\psi^{\prime\prime\prime}_{j}=0,\\
			\psi_j(0)=\psi_j(L)=\psi^{\prime}_{j}(0)=\psi^{\prime}_{j}(L)=0.
		\end{cases}
	\end{align*}
	Since $L\notin \mathcal{N}$, from \cite[Lemma 3.5]{Rosier} it follows that
	\begin{align*}
		\psi_j=0,\ j=1,...,m \implies \varphi_j=\frac{1}{m}\psi,\ j=1,...,m.
	\end{align*}
	Consequently,
$$	
		\varphi_j=\varphi_1,\ j=1,...,m.
$$

	Now define
	\begin{align*}
		\theta=\sum_{j=m+1}^N\varphi_j\text{ \ \ and \ \ }\theta_j=\theta-(N-m)\varphi_j,\ j=m+1,...,N.
	\end{align*}
	Note that, for each $j \in \{m+1,...,N\}$, $\theta_j$ solves the problem
	\begin{align*}
		\begin{cases}
			\lambda \theta_j+\theta^{\prime}_{j}+\theta^{\prime\prime\prime}_{j}=0,\\
			\theta_j(0)=\theta_j(L)=\theta^{\prime}_{j}(0)=\theta^{\prime\prime}_{j}(L)=0.
		\end{cases}
	\end{align*}
	Additionally, $L\notin \mathcal{N}^*$, from \cite[Propositions 1 and 2]{GG1} it follows that
	\begin{align*}
		\theta_j=0,\ j=m+1,...,N
	\end{align*}
	so that
	\begin{align*}
		\varphi_j=\frac{1}{N-m}\theta,\ j=m+1,...,N.
	\end{align*}
	Consequently,
	\begin{align}\label{3.50}
		\varphi_j=\varphi_N,\ j=m+1,...,N.
	\end{align}
This implies that, $\varphi_{1}$ and $\varphi_{N}$ satisfy the spectral problem
    \begin{align*}
	\begin{cases}
		\lambda\varphi_j+\varphi^{\prime}_{j}+\varphi^{\prime\prime\prime}_{j}=0,&x \in (0,l_j),\ j=1,N,\\
		\varphi_N(0)=\varphi_1(0),&\\
		m\varphi^{\prime\prime}_{1}(0)+(N-m)\varphi^{\prime\prime}_{N}(0)=0,&\\		
		\varphi_j(L)=\varphi^{\prime}_{j}(0)=0,&j=1,N,\\
		\varphi^{\prime}_{1}(L)=\varphi^{\prime\prime}_{N}(L)=0,&
	\end{cases}
\end{align*}
which corresponds to \eqref{spectral problem-2} in the case $\varsigma_{Ne}=m$ and $\varsigma_{Di}=N-m$. Therefore, the unique solution is $\varphi_{1}=\varphi_{N}=0$. Finally, as $0=\varphi_{1}=\varphi_{j}$ for $j=1,\dots,m$ and $0=\varphi_{N}=\varphi_{j}$ for $j=m+1,\dots,N$,  we obtain $\varphi=0$.\end{proof}

To conclude this part, we analyze the case of a full Neumann or Dirichlet boundary, i.e., we consider system \eqref{LkdV-2} with either $N$ Neumann or $N$ Dirichlet controls.
\subsubsection{\textbf{Full Neumann}} We start by analyzing the associate spectral problem arising for $N$ Neumann controls 
\begin{align}\label{spectral problem-Ne}
\begin{cases}			
\lambda\varphi_j+\varphi^{\prime}_{j}+\varphi^{\prime\prime\prime}_{j}=0,&x \in (0,l_j),\ j=1,...,N,\\
\varphi_j(0)=\varphi_1(0),&j=1,...,N,\\
\displaystyle\sum_{j=1}^N\varphi^{\prime\prime}_{j}(0)=0,&\\		
\varphi_j(l_j)=\varphi^{\prime}_{j}(0)=0,&j=1,...,N,\\
\varphi^{\prime}_{j}(l_j)=0,&j=1,...,N.
\end{cases}
\end{align}
For this case, we have the following result.
\begin{lemma}\label{lemma 3.9}
Let $L>0$ and assume $l_j=L$ for $j=1,...,N$. There exist $\lambda\in \mathbb{C}$ and $\varphi\in \mathbb{H}^3(\mathcal{T})\backslash\{0\}$ satisfying \eqref{spectral problem-Ne} if and only if $L \in \mathcal{N}\cup \mathcal{N}^*$.
\end{lemma}
\begin{proof}
Consider the function $\psi=\displaystyle\sum_{j=1}^{N}\varphi_{j}$. We will see that if the length is not critical, it is enough to analyze the spectral problem satisfied by the sum function.
\begin{claim}
If $L\notin \mathcal{N}$, then $\psi\equiv0$ if and only if $\varphi_{j}\equiv0$ for all $j=1,\dots, N$. 
\end{claim}
\begin{proof}
One direction is direct. In the other hand, if  $\psi\equiv0$, we have $0=\psi(0)=N\varphi_{j}(0)$, thus $\varphi_{j}$ solves
\[\begin{cases}
\lambda \varphi_j+\varphi^{\prime}_{j}+\varphi^{\prime}_{j}=0,\\
\varphi_j(0)=\varphi_j(L)=\varphi^{\prime}_{j}(0)=\varphi^{\prime}_{j}(L)=0.   
\end{cases}\]
Since $L\notin \mathcal{N}$, from \cite[Lemma 3.5]{Rosier} it follows that $\varphi_{j}\equiv 0$.
\end{proof}
The condition $L\notin \mathcal{N}$ is necessary. In fact, if not, we can not control $\psi_{j}=\varphi_{j}-\dfrac{1}{N}\psi$. By the previous claim, it is enough to study when we have a non-trivial function $\psi$. We can immediately see that $\psi$ solves

\[\begin{cases}
\lambda \psi+\psi^{\prime}+\psi^{\prime\prime\prime}=0,\\
\psi(L)=\psi^{\prime}(0)=\psi^{\prime}(L)=\psi^{\prime\prime}(0)=0.   
\end{cases}\]

Considering $\theta(x)=\psi(L-x)$ and $\tilde{\lambda}=-\lambda$, we get

\begin{align}
\label{eq: theta}
\begin{cases}
\tilde{\lambda}\theta+\theta^{\prime}+\theta^{\prime\prime\prime}=0,\\
\theta(0)=\theta^{\prime}(L)=\theta^{\prime}(0)=\theta^{\prime\prime}(L)=0.   
\end{cases}
\end{align}
Define $P_{\Tilde{\lambda}}(\mu)=\mu^{3}+\mu+\Tilde{\lambda}$.
\begin{claim}\label{claim 2.4} If $P_{\Tilde{\lambda}}$ admits multiples roots, then $\theta=0$.
\end{claim}
\begin{proof}
We follow the proof of Claim~\ref{claim 2.3}. Multiplying the system \eqref{eq: theta} by $e^{-ix\xi}$, integrating by parts in $(0,L)$ and using the boundary conditions we get, for every $\xi \in \mathbb{C}$
\[(\xi^{3}-\xi-p)\hat{\theta}(\xi)=i\kappa(1-\xi^{2})e^{-iL\xi}+\beta,\]
where $p=i\Tilde{\lambda}$, $\kappa=-\theta(L)$ and $\beta=\theta^{\prime\prime}(0)$. If $P_{\Tilde{\lambda}}$ has a triple root, then by Girard's relations $\tilde{\lambda}=0$, thus $p=0$ and $\hat{\theta}$ satisfies 
\[(\xi^{3}-\xi)\hat{\theta}(\xi)=i\kappa(1-\xi^{2})e^{-iL\xi}+\beta, \qquad \forall \xi \in \mathbb{C}.\]
Evaluating in $\xi=1$, we get $\theta^{\prime\prime}(0)=\beta=0$. Therefore, $\theta$ satisfies $\theta(0)=\theta^{\prime}(0)=\theta^{\prime\prime}(0)=0$, which implies $\theta=0$. If $P_{\Tilde{\lambda}}$ has a double root $\mu_{1}=\mu_{2}=\sigma$, by Girard's relations $\sigma=\pm\dfrac{i}{\sqrt{3}}$ and the solution $\theta$ of \eqref{eq: theta} can be written as
\[\theta(x)=d_{0}e^{-2\sigma x}+d_{1}e^{\sigma x}+d_{2}xe^{\sigma x}.\]
The boundary condition $\theta(0)=0$ gives us $d_{0}+d_{1}=0$. While, $\theta^{\prime}(0)=\theta^{\prime}(L)$ give us respectively
\[-2\sigma d_{0}+d_{1}\sigma+d_2=0\quad \text{and}\quad -2\sigma d_{0}e^{-2\sigma L}+d_{1}\sigma e^{\sigma L}+d_{2}(e^{\sigma L}+L\sigma e^{\sigma L})=0.\]
Using $d_{0}=-d_{1}$, we get
\[\begin{cases}
3\sigma d_{1}+d_{2}=0\\
\sigma d_{1}(2e^{-3\sigma L}+1)+d_{2}(1+\sigma L)=0.
\end{cases}\]
The previous system has a non-trivial solution if and only if $\sigma L=2\sigma e^{-3\sigma L}$. In that case $d_{2}=-3\sigma d_{1}$ On the other hand, the condition $\theta^{\prime\prime}(L)=0$ provides that
\begin{align*}
	d_1\sigma+d_2(2+L\sigma)=-4d_0\sigma e^{-3\sigma L}.
\end{align*}
Using $d_{0}=-d_{1}$, $\sigma L=2\sigma e^{-3\sigma L}$ and $d_{2}=-3\sigma d_{1}$ we obtain $-4=5\sigma L$, which is not possible since $L>0$ and $\sigma \in i\mathbb{R}$, and Claim \ref{claim 2.4} follows.
\end{proof}
By the previous claim, $\theta(x)=d_{0}e^{\mu_{0}x}+d_{1}e^{\mu_{1}x}+d_{2}e^{\mu_{2}x}$, where $\mu_{0}$, $\mu_{1}$ and $\mu_{2}$ are the simple roots of the characteristic polynomial $\mu^{3}+\mu+\tilde{\lambda}=0$. By imposing the boundary conditions, we deduce the following overdetermined system 
\[\begin{cases}
d_{0}+d_{1}+d_{2}=0\\
\mu_{0}d_{0}+\mu_{1}d_{1}+\mu_{2}d_{2}=0\\
\mu_{0}d_{0}e^{\mu_{0}L}+\mu_{1}d_{1}e^{\mu_{1}L}+\mu_{2}d_{2}e^{\mu_{2}L}=0\\
\mu_{0}^{2}d_{0}e^{\mu_{0}L}+\mu_{1}^{2}d_{1}e^{\mu_{1}L}+\mu_{2}^{2}d_{2}e^{\mu_{2}L}=0.
\end{cases}\]
By calling $a=\mu_{0}L$, $b=\mu_{1}L$ and $c=\mu_{2}L$, we obtain the following matrix system
\[\begin{bmatrix}
    1&1&1\\
    ae^a&be^b&ce^c\\
    a&b&c\\
    a^{2}e^a&b^{2}e^b&c^{2}e^c
\end{bmatrix}\begin{bmatrix}
C_{1}\\C_{2}\\C_{3}
\end{bmatrix}=\begin{bmatrix}
0\\0\\0\\0
\end{bmatrix}.\]
By Gauss-Jordan elimination
\[\begin{bmatrix}
    1&1&1\\
    ae^a&be^b&ce^c\\
    a&b&c\\
    a^{2}e^a&b^{2}e^b&c^{2}e^c
\end{bmatrix}
\sim\begin{bmatrix}
   1&1&1\\
    ae^a&be^b&ce^c\\
    0&b-a&c-a\\
    0&0&(c-a)(ce^c-be^{b})
\end{bmatrix}=
\begin{bmatrix}
    &A_{a,b,c}&\\
    0&0&B_{a,b,c}
\end{bmatrix}.\]
We have a non-trivial function $\theta$ if and only if
\begin{itemize}
    \item The rank of $A_{a,b,c}$ is one.
    \item The rank of $A_{a,b,c}$ is two and $B_{a,b,c}=0$.
\end{itemize}
Easy computations give us 
\begin{itemize}
    \item $rank(A_{a,b,c})=1$, if and only if $a=b=c$. This case is not possible, because the roots are simple.
    \item $rank(A_{a,b,c})=2$
    \begin{enumerate}
        \item $a=b\neq c$, $be^{b}=ce^{c}$. Not possible.
        \item $a=b$, $be^{b}\neq ce^{a}$. Not possible.
        \item  $ae^{a}\neq be^{b}$, $ce^{c}(a-b)+be^{b}(c-a)+ae^{a}(b-c)=0$.
        \item $a\neq b$, $ae^{a}=be^{b}=ce^{c}$.
    \end{enumerate}
\end{itemize}
On the other hand, again the roots are simple $B_{a,b,c}=0$ if and only if $be^{b}=ce^{c}$. Using this in the third case, we get
\[0=ce^{c}(a-b)+be^{b}(c-a)+ae^{a}(b-c)=(b-c)(ae^{a}-ce^{c}),\]
which implies $ae^{a}=ce^{c}=be^{b}$, that contradicts the third case. Finally, in the fourth case, we have a non-trivial solution if $ae^{a}=be^{b}=ce^{c}$. Recall that Girard's relations $a+b+c=0$ and $L^{2}=-(a^{2}+ab+b^{2})$, which correspond to the expression \eqref{critical-2} of the critical lengths $\mathcal{N}^{*}$.
\end{proof}
\begin{remark}
In the previous analysis, we have shown that different spectral problems could have the same set of critical lengths. $\mathcal{N}^{*}$ is the set of critical lengths of the spectral problem \eqref{eq: theta} which is a different one than \eqref{eq: critical dir} studied in \cite{GG1}.

\end{remark}
\subsubsection{\textbf{Full Dirichlet}}
We pass now to the full Dirichlet. By Lemma~\ref{lemma 3.5}, it is enough to focus on the spectral problem \eqref{spectral problem}.
\begin{align}\label{spectral problem-Di}
\begin{cases}			
\lambda\varphi_j+\varphi^{\prime}_{j}+\varphi^{\prime\prime\prime}_{j}=0,&x \in (0,l_j),\ j=1,...,N,\\
\varphi_j(0)=\varphi_1(0),&j=1,...,N,\\
\displaystyle\sum_{j=1}^N\varphi^{\prime\prime}_{j}(0)=0,&\\		
\varphi_j(l_j)=\varphi^{\prime}_{j}(0)=0,&j=1,...,N,\\
\varphi^{\prime\prime}_{j}(l_j)=0,&j=1,...,N.
\end{cases}
\end{align}
For this case, we have the following result.
\begin{lemma}\label{lemma 3.10}
Let $L>0$ and assume $l_j=L$ for $j=1,...,N$. There exist $\lambda\in \mathbb{C}$ and $\varphi\in \mathbb{H}^3(\mathcal{T})\backslash\{0\}$ satisfying \eqref{spectral problem-Di} if and only if $L \in \mathcal{N}^{*}\cup \mathcal{N}^{\dagger}$.
\end{lemma}
As in the full Neumann case, it is enough to analyze the spectral problem associated with the sum function $\psi=\displaystyle\sum_{j=1}^{N}\varphi_{j}$. 
\begin{claim}
If $L\notin \mathcal{N}^{*}$, then $\psi\equiv0$ if and only if $\varphi_{j}\equiv0$ for all $j=1,\dots, N$. 
\end{claim}
\begin{proof}
One direction is direct. In the other hand, if  $\psi\equiv0$, we have $0=\psi(0)=N\varphi_{j}(0)$, thus $\varphi_{j}$ solves
\[\begin{cases}
\lambda \varphi_j+\varphi^{\prime}_{j}+\varphi^{\prime\prime}_{j}=0,\\
\varphi_j(0)=\varphi_j(L)=\varphi^{\prime}_{j}(0)=\varphi^{\prime}_{j}(L)=0.   
\end{cases}\]
			Since $L\notin \mathcal{N}^{*}$, from \cite[Proposition 1 and 2]{GG1} it follows that $\varphi_{j}\equiv 0$.
\end{proof}

We can immediately see that $\psi$ solves

\[\begin{cases}
\lambda \psi+\psi^{\prime}+\psi^{\prime\prime\prime}=0,\\
\psi(L)=\psi^{\prime}(0)=\psi^{\prime\prime}(L)=\psi^{\prime\prime}(0)=0.   
\end{cases}\]

Considering $\theta(x)=\psi(L-x)$ and $\tilde{\lambda}=-\lambda$, we get

\begin{align}
\label{eq: theta dir}
\begin{cases}
\tilde{\lambda}\theta+\theta^{\prime}+\theta^{\prime\prime}=0,\\
\theta(0)=\theta^{\prime}(L)=\theta^{\prime\prime}(0)=\theta^{\prime\prime}(L)=0.   
\end{cases}    
\end{align}
Define $P_{\Tilde{\lambda}}(\mu)=\mu^{3}+\mu+\Tilde{\lambda}$.
\begin{claim}\label{claim 2.5} If $P_{\Tilde{\lambda}}$ admits multiples roots, then $\theta=0$.
\end{claim}
\begin{proof}
We follow the proof of Claim~\ref{claim 2.3}. Multiplying the system \eqref{eq: theta dir} by $e^{-ix\xi}$, integrating by parts in $(0,L)$ and using the boundary conditions we get, for every $\xi \in \mathbb{C}$
\[(\xi^{3}-\xi-p)\hat{\theta}(\xi)=i\kappa(1-\xi^{2})e^{-iL\xi}+\iota\xi,\]
where $p=i\Tilde{\lambda}$, $\kappa=-\theta(L)$ and $\iota=i\theta^{\prime}(0)$. If $P_{\Tilde{\lambda}}$ has a triple root, then by Girard's relations $\tilde{\lambda}=0$, thus $p=0$ and $\hat{\theta}$ satisfies 
\[(\xi^{3}-\xi)\hat{\theta}(\xi)=i\kappa(1-\xi^{2})e^{-iL\xi}+\iota\xi, \qquad \forall \xi \in \mathbb{C}.\]
Evaluating in $\xi=1$, we get $\theta^{\prime}(0)=\iota=0$. Therefore, $\theta$ satisfies $\theta(0)=\theta^{\prime}(0)=\theta^{\prime\prime}(0)=0$, which implies $\theta=0$.
\\

\noindent If $P_{\Tilde{\lambda}}$ has a double root $\mu_{1}=\mu_{2}=\sigma$, by Girard's relations $\sigma=\pm\dfrac{i}{\sqrt{3}}$ and the solution $\theta$ of \eqref{eq: theta dir} can be written as
\[\theta(x)=d_{0}e^{-2\sigma x}+d_{1}e^{\sigma x}+d_{2}xe^{\sigma x}.\]
The boundary conditions $\theta(0)=\theta^{\prime}(L)=0$ gives us respectively
\[d_0+d_1=0, \quad \text{and} \quad -2\sigma d_{0}+d_{1}\sigma+d_2=0,\]
from where we get $d_{0}=-d_{1}$ and $d_{2}=-3\sigma d_{1}$. On the other hand, the boundary condition $\theta^{\prime\prime}(0)$ gives
\[4d_0\sigma^2+d_1\sigma^2+2d_2\sigma=0.\]
Replacing $d_{0}=-d_{1}$ and $d_{2}=-3\sigma d_{1}$ in the above expression we get $d_{1}=0$, and finally $\theta=0$, showing the Claim \ref{claim 2.5}.
\end{proof}
By the previous claim, $\theta(x)=d_{0}e^{\mu_{0}x}+d_{1}e^{\mu_{1}x}+d_{2}e^{\mu_{2}x}$, where $\mu_{0}$, $\mu_{1}$ and $\mu_{2}$ are the simple roots of the characteristic polynomial $\mu^{3}+\mu+\tilde{\lambda}=0$. By imposing the boundary conditions and  calling $a=\mu_{0}L$, $b=\mu_{1}L$ and $c=\mu_{2}L$, we obtain the following matrix system
\[\begin{bmatrix}
    1&1&1\\
     ae^a&be^b&ce^c\\
    a^2&b^2&c^2\\
    a^{2}e^a&b^{2}e^b&c^{2}e^c
\end{bmatrix}\begin{bmatrix}
C_{1}\\C_{2}\\C_{3}
\end{bmatrix}=\begin{bmatrix}
0\\0\\0\\0
\end{bmatrix}.\]
By Gauss-Jordan elimination, $c\neq 0$
\[\begin{bmatrix}
    1&1&1\\
    ae^a&be^b&ce^c\\
    a&b&c\\
    a^{2}e^a&b^{2}e^b&c^{2}e^c
\end{bmatrix}\thicksim\begin{bmatrix}
    1&1&1\\
    ae^a&be^b&ce^c\\
    0&b^2-a^2&c^2-a^2\\
    0&0&\dfrac{(c-a)}{c}(c^2e^c-b^2e^{b})
\end{bmatrix}=\begin{bmatrix}
    &\Tilde{A}_{a,b,c}&\\
    0&0&\Tilde{B}_{a,b,c}
\end{bmatrix}.\]
Easy computations give us 
\begin{itemize}
    \item $rank(\tilde{A}_{a,b,c})=1$, if and only if
    \[ce^c-ae^a=0, \ a^2-c^2=0, ce^c-be^b=0, \  b^2-c^2=0.\]
    Not possible.
    \item $rank(\tilde{A}_{a,b,c})=2$
    \begin{enumerate}
        \item  $ce^c-ae^a=0$, $a^2-b^2=0$, $ce^c-be^b=0$, $b^2-c^2\neq0.$ Not possible.
        \item $be^b-ae^a=0$, $a^2-b^2=0$, $ce^c-be^b\neq0$. 
        \item $ce^c-ae^a=0$, $a^2-b^2\neq0$, $ce^c-be^b=0$. Not possible, because $L\notin \mathcal{N}^*$.
        \item $a^2e^a(b-c)+b^2e^b(c-a)+c^{2}e^c(a-b)=0$.
    \end{enumerate}
\end{itemize}
On the other hand, by symmetry
\[\Tilde{B}_{a,b,c}=\dfrac{(c-a)(c^{2}e^{c}-b^{2}e^{b})}{c}=0, \quad \Longrightarrow c^{2}e^{c}=b^{2}e^{b}=a^{2}e^{a}.\]
Then we have a non-trivial solution if $c^{2}e^{c}=b^{2}e^{b}=a^{2}e^{a}$. Recall that Girard's relations $a+b+c=0$ and $L^{2}=-(a^{2}+ab+b^{2})$, which corresponds to the expression \eqref{critical-3} of the critical lengths $\mathcal{N}^{\dagger}$.


\subsection{Proof of Theorem \ref{main-a}} The main result in this work, Theorem \ref{main-a}, is a consequence of Lemmas \ref{lemma 3.4} and \ref{lemma 3.5}, combined with the results obtained in the study of each case in \eqref{cases}, through Lemmas \ref{lemma 3.6}, \ref{lemma 3.6-b}, \ref{lemma 3.7}, \ref{lemma 3.8}, \ref{lemma 3.9} and \ref{lemma 3.10}.


\appendix
\section{A review of trace estimates}\label{trace-review}
Let us review some results for the single KdV equation to see it. Precisely, the first result is given by \cite[Propositions 2.7-2.9]{Bona Sun Zhang} and \cite[Proposition 2.2]{Capistrano 2017}.
\begin{proposition}\label{W.P.KdV(0,L)-1} Let $L,T>0$ be given. For any $(h_1,h_2,h_3)\in H^\frac{1}{3}(0,T)\times H^\frac{1}{3}(0,T)\times L^2(0,T)$ there exists a unique solution $v \in \mathcal{Z}_T:=C([0,T];L^2(0,L))\cap L^2(0,L;H^1(0,L))$ to the problem
\begin{align}\label{KdV(0,L)-1}
\begin{cases}
\partial_{t}v+\partial_{x}v+\partial_{x}^{3}v=0,&t\in(0,T),\ x \in (0,L),\\
v(t,0)=h_1(t),\ \ v(t,L)=h_2(t),\ \ \partial_{x}v(t,L)=h_3(t),&t \in(0,T),\\
v(0,x)=0,&x \in (0,L)
\end{cases}
\end{align}
which possesses the sharp Kato smoothing property
\begin{align*}
\partial_x^kv\in L_x^\infty\left(0,L;H^{\frac{1-k}{3}}(0,T)\right),\ k=0,1,2.
\end{align*}
	Furthermore, there exists $C>0$ such that
	\begin{align*}
		\|v\|_{\mathcal{Z}_T}+\sum_{k=0}^2\sup_{x \in (0,L)}\|\partial_x^kv(\cdot,x)\|_{H^\frac{1-k}{3}(0,T)}\leq C\left(\|h_1\|_{H^\frac{1}{3}(0,T)}+\|h_2\|_{H^\frac{1}{3}(0,T)}+\|h_3\|_{L^2(0,T)}\right).
	\end{align*}
\end{proposition}
Henceforth we denote by $W_b(t)\vec{h}$ the solution of \eqref{KdV(0,L)-1} corresponding to $\vec{h}=(h_1,h_2,h_3)$. Considering this, \cite[Proposition 2.1]{Bona Sun Zhang} gives us the following proposition.
\begin{proposition} Let $L,T>0$ be given. For any $v_0 \in L^2(0,L)$ there exists a unique solution $v \in \mathcal{Z}_T$ for the problem
\begin{align}\label{KdV(0,L)-2}
\begin{cases}
\partial_{t}v+\partial_{x}v+\partial_{x}^{3}v=0,&t\in(0,T),\ x \in (0,L),\\
v(t,0)=0,\ \ v(t,L)=0,\ \ \partial_{x}v(t,L)=0,&t \in(0,T),\\
v(0,x)=v_0(x),&x \in (0,L),
\end{cases}
\end{align}
which is given by $v=W(\cdot)v_0$ where $\{W(t)\}_{t\geq 0}$ is the $C_0$-semigroup of contractions generated in $L^2(0,L)$ by the operator $A_Lv=-\partial_{x}v-\partial_{x}^{3}v$ with domain
	\begin{align*}
		D(A_L)=\{w\in H^3(0,L);\ v(0)=v(L)=\partial_{x}v(L)=0\}.
	\end{align*}
	Moreover, there exists a constant $C>0$ such that
	\begin{align*}
		\|v\|_{\mathcal{Z}_T}\leq C\|v_0\|_{L^2(0,L)}\quad \text{and}\quad \|\partial_{x}v(\cdot,0)\|_{L^2(0,T)}\leq C\|v_0\|_{L^2(0,L)}.
	\end{align*}
\end{proposition}

We will show that the solution $v=W(\cdot)v_0$ of \eqref{KdV(0,L)-2} possesses the Kato smoothing property. According to \cite{Krammer Rivas Zhang}, the linear KdV equation
\begin{align*}
\begin{cases}
		\partial_{t}z+\partial_{x}z+\partial_{x}^{3}z=0,&\ t \in \mathbb{R}_+,\ x \in \mathbb{R}\\
		z(0,x)=z_0(x)
\end{cases}
\end{align*}
has a unique solution, given by
\begin{align*}
	z(t,x)=(W_\mathbb{R}z_0)(x):=c\int_\mathbb{R}e^{i(\xi^3-\xi)t}e^{ix\xi}\hat{z_0}(\xi)d\xi
\end{align*}
where $c \in \mathbb{R}$ and $\Hat{z_0}$ denotes the Fourier transform of $z_0$. The result can be read below. 
\begin{proposition}\cite[Lemma 2.9]{Krammer Rivas Zhang}\ There exists a constant $C>0$ such that, for every $z_0 \in L^2(\mathbb{R})$,
	\begin{align*}
		\sum_{k=0}^{2}\sup_{x \in \mathbb{R}}\|\partial_x^kW_\mathbb{R}(\cdot)v_0(x)\|_{H_t^\frac{1-k}{3}(\mathbb{R})}\leq C\|v_0\|_{L^2(\mathbb{R})}.
	\end{align*}
\end{proposition}
Additionally, \cite{Bona Sun Zhang} showed the following:
\begin{proposition}\cite[Proposition 2.13]{Bona Sun Zhang}\label{trace-KdV-(0,L)2} Let $L,T>0$ be given. There exists a constant $C>0$ such that, for every $v_0 \in L^2(0,L)$, the solution $v=W(\cdot)v_0$ of \eqref{KdV(0,L)-2} satisfies
	\begin{align*}
		\sum_{k=0}^{2}\sup_{x \in (0,L)}\|\partial_x^kv(\cdot,x)\|_{H^\frac{1-k}{3}(0,T)}\leq C\|v_0\|_{L^2(0,L)}.
	\end{align*}
\end{proposition}
Combining the previous propositions, precisely, Propositions \ref{W.P.KdV(0,L)-1} and \ref{trace-KdV-(0,L)2}, and \cite[Proposition 2.6]{Capistrano 2017}, the next results are verified.
\begin{proposition}
	Consider $L,T>0$. For every $v_0 \in L^2(0,L)$ and $(h_1,h_2,h_3)\in \mathcal{H}_T:=H^\frac{1}{3}(0,T)\times H^\frac{1}{3}(0,T)\times L^2(0,T)$ there exists a unique solution $v \in \mathcal{Z}_T:=C([0,T];L^2(0,L))\cap L^2(0,L;H^1(0,L))$ to the problem
	\begin{align*}
		\begin{cases}
			\partial_{t}v+\partial_{x}v+\partial_{x}^{3}v=0,&t\in(0,T),\ x \in (0,L),\\
			v(t,0)=h_1(t),\ \ v(t,L)=h_2(t),\ \ \partial_{x}v(t,L)=h_3(t),&t \in(0,T),\\
			v(0,x)=v_0(x),&x \in (0,L),
		\end{cases}
	\end{align*}
	which possesses the sharp Kato smoothing property
	\begin{align*}
		\partial_x^kv\in L_x^\infty\left(0,L;H^{\frac{1-k}{3}}(0,T)\right),\ k=0,1,2.
	\end{align*}
	Furthermore, there exists $C>0$ such that
	\begin{align*}
		\|v\|_{\mathcal{Z}_T}+\sum_{k=0}^2\sup_{x \in (0,L)}\|\partial_x^kv(\cdot,x)\|_{H^\frac{1-k}{3}(0,T)}\leq C \left(\|v_0\|_{L^2(0,L)}+\|h\|_{\mathcal{H}_T}\right).
	\end{align*}
\end{proposition}
\begin{corollary}\label{W.P.KdV(0,L)-4}
	Let $T,L>0$ be given. Given $v^T \in L^2(0,L)$ and $h \in H^{\frac{1}{3}}(0,T)$ there exists a unique solution $v \in \mathcal{Z}_T$ for the problem
	\begin{align*}
		\begin{cases}
			\partial_{t}v+\partial_{x}v+\partial_{x}^{3}v=0,&t\in(0,T),\ x \in (0,L),\\
			v(t,0)=h(t),\ \ v(t,L)=\partial_{x}v(t,0)=0,&t \in(0,T),\\
			v(T,x)=v^T(x),&x \in (0,L)
		\end{cases}
	\end{align*}
	which satisfies
	\begin{align*}
		\partial_x^kv\in L_x^\infty\left(0,L;H^{\frac{1-k}{3}}(0,T)\right),\ k=0,1,2.
	\end{align*}
	and
	\begin{align*}
		\|v\|_{\mathcal{Z}_T}+\sum_{k=0}^2\sup_{x \in (0,L)}\|\partial_x^kv(\cdot,x)\|_{H^\frac{1-k}{3}(0,T)}\leq C \left(\|v^T\|_{L^2(0,L)}+\|h\|_{H^\frac{1}{3}(0,T)}\right),
	\end{align*}
	for some constant $C>0$.
\end{corollary}
\begin{proposition}\label{W,P.Kdv(0,L)-5}
	Given $T,L>0$, $v^T \in L^2(0,L)$ and $h \in H^{-\frac{1}{3}}(0,T)$ there exists a unique solution $v \in \mathcal{Z}_T$ of the problem
	\begin{align*}
		\begin{cases}
			\partial_{t}v+\partial_{x}v+\partial_{x}^{3}v=0,&t\in(0,T),\ x \in (0,L),\\
			v(t,L)=\partial_{x}v(t,0)=0,\ \ \partial_{x}^{2}v(t,0)=h(t),&t \in(0,T),\\
			v(T,x)=v^T(x),&x \in (0,L)
		\end{cases}
	\end{align*}
	which satisfies
	\begin{align*}
		\partial_x^kv\in L_x^\infty\left(0,L;H^{\frac{1-k}{3}}(0,T)\right),\ k=0,1,2.
	\end{align*}
	and
	\begin{align*}
		\|v\|_{\mathcal{Z}_T}+\sum_{k=0}^2\sup_{x \in (0,L)}\|\partial_x^kv(\cdot,x)\|_{H^\frac{1-k}{3}(0,T)}\leq C \left(\|v^T\|_{L^2(0,L)}+\|h\|_{H^{-\frac{1}{3}}(0,T)}\right).
	\end{align*}
	for some positive constant $C$.
\end{proposition}

We finish this appendix with an auxiliary result about the KdV operator.

\begin{lemma}\label{Ay in H^-2}
For any $L>0$, if $y \in H^1(0,L)$ then $-\partial_{x}y-\partial_{x}^{3}y\in H^{-2}(0,L)$ with
	\begin{align*}
		\|-\partial_{x}y-\partial_{x}^{3}y\|_{H^{-2}(0,L)}\leq \|y\|_{H^1(0,L)}.
	\end{align*}
\end{lemma}

\begin{proof}
	We know that $H^{-2}(0,L)$ has the following characterization:
	\begin{align*}
		H^{-2}(0,L)=\left\{T \in \mathcal{D}'(0,L);\ T=g_0+\partial_{x}g_{1}+\partial_{x}^{2}g_{2},\ g_0,g_1,g_2\in L^2(0,L)\right\}
	\end{align*}
	where the derivatives are taken in the distributional sense. Moreover,
	\begin{align*}
		\|T\|_{H^{-2}(0,L)}=\inf\left\{\left(\sum_{j=1}^2\|g_i\|_{L^2(0,L)}^2\right)^\frac{1}{2};\ g_0,g_1,g_2\in L^2(0,L)\text{ and }T=g_0+\partial_{x}g_{1}+\partial_{x}^{2}g_{2}\right\}.
	\end{align*}
	
	Suppose that $y \in H^1(0,L)$. Defining $g_0=0$, $g_1=-y$ and $g_2=-\partial_{x}y$ we have $g_0,g_1,g_2 \in L^2(0,L)$ and
	\begin{align*}
		-\partial_{x}y-\partial_{x}^{3}y=g_0+\partial_{x}g_{1}+\partial_{x}^{2}g_{2}
	\end{align*}
	so $-\partial_{x}y-\partial_{x}^{3}y\in H^{-2}(0,L)$. Furthermore,
	\begin{align*}
		\|-\partial_{x}y-\partial_{x}^{3}y\|_{H^{-2}(0,L)}\leq \left(\sum_{j=1}^2\|g_i\|_{L^2(0,L)}^2\right)^\frac{1}{2}=\left(\|y\|_{L^2(0,L)}^2+\|\partial_{x}y\|_{L^2(0,L)}^2\right)^\frac{1}{2}=\|y\|_{H^1(0,L)},
	\end{align*}
	giving the result.
\end{proof}


\section{Controllability of Neumann and Dirichlet conditions}\label{control-ND}
From now on, $\alpha=N$. Fixed $m=0,\dots,N$, we will consider Neumann boundary controls on the first $m$ edges and Dirichlet boundary controls on the remaining ones. We only analyze the reachable states from the origin, that is, we will consider $u^0_j=0$ (the case where $u_0 \in \mathbb{L}^2(\mathcal{T})$ is arbitrary, and $u_0=0$ is done similarly and leads to the same observability inequality). In these terms, we have the following characterization of the controllability.

\begin{lemma}\label{contr. lemma}
	For $T>0$, the controls $g_j \in L^2(0,T)$, $j=1,\dots,m$ and $p_j\in L^2(0,T)$, $j=m+1,\dots,N$, drive $u^0_j=0$ to $u^T\in \mathbb{L}^2(\mathcal{T})$ if and only if
	\begin{align}\label{Char. Contr.}		\sum_{j=1}^N\int_0^{l_j}u_j^T\varphi_j^T=\sum_{j=1}^m\int_0^T\partial_{x}\varphi_{j}(t,l_j)g_j(t)-\sum_{j=m+1}^N\int_0^T\partial_{x}^{2}\varphi_{j}(t,l_j)p_j(t),
	\end{align}
 for any $\varphi^T=(\varphi_1^T,...,\varphi_N^T) \in \mathbb{L}^2(\mathcal{T})$ and $\varphi$ solution of the system \eqref{Adj} associated to $\varphi^T$.
\end{lemma}
\begin{proof}
	Let $u$ be the solution of \eqref{LkdV-2} corresponding to $u^0_j$. Given $\varphi^T \in \mathbb{L}^2(\mathcal{T})$, multiplying the first equation in \eqref{LkdV-2} by the solution $\varphi$ of \eqref{Adj}, integrating by parts, and using the boundary conditions, we obtain
	\begin{align*}
		\sum_{j=1}^N\int_0^{l_j}u_j(T,x)\varphi_j^T-\sum_{j=1}^N\int_0^{l_j}u_j^0(x)\varphi_j(0,x)&=\sum_{j=1}^m\int_0^T\partial_{x}\varphi_{j}(t,l_j)g_j(t)-\sum_{j=m+1}^N\int_0^T\partial_{x}^{2}\varphi_{j}(t,l_j)p_j(t)
	\end{align*}
	so, for $u^0_j=0$ we have that
	\begin{align}\label{eq 3.3}
		\sum_{j=1}^N\int_0^{l_j}u_j(T,x)\varphi_j^T(x)=\sum_{j=1}^m\int_0^T\partial_{x}\varphi_{j}(t,l_j)g_j(t)-\sum_{j=m+1}^N\int_0^T\partial_{x}^{2}\varphi_{j}(t,l_j)p_j(t), \forall\varphi^T \in \mathbb{L}^2(\mathcal{T}).
	\end{align}
	If $u(T,\cdot)=u^T$ then \eqref{Char. Contr.} immediately follows from \eqref{eq 3.3}. Conversely, if \eqref{Char. Contr.} holds, thanks to \eqref{eq 3.3} holds that
	\begin{align*}
		(u(T,\cdot),\varphi^T)_{\mathbb{L}^2(\mathcal{T})}=(u^T,\varphi^T)_{\mathbb{L}^2(\mathcal{T})},\ \forall \varphi^T \in \mathbb{L}^2(\mathcal{T}),
	\end{align*}
	and  consequently $u(T,\cdot)=u^T$, showing the lemma.
\end{proof}

Consider the bilinear map $B:\mathbb{L}^2(\mathcal{T})\times \mathbb{L}^2(\mathcal{T})\rightarrow \mathbb{R}$ given by
\begin{align*}
	B(\varphi^T,\psi^T)=\sum_{j=1}^{N}\int_0^Tu_j(T,x)\psi_j^T(x),
\end{align*}
with $u=(u_1,...,u_N)$ being the solution of the system \eqref{LkdV-2} with boundary controls $p_j=-\partial_{x}^{2}\varphi_{j}(\cdot,l_j)$ and $g_j=\partial_{x}\varphi_{j}(\cdot,l_j)$, where $\varphi(t,\cdot)=S(T-t)\varphi^T$. Let us prove some properties of the bilinear map $B$. 

\vspace{0.2cm}
\noindent\textbf{(i) B is continuous.}
\vspace{0.2cm}

Using the multiplier method, we obtain
\begin{align*}
	B(\varphi^T,\psi^T)=\sum_{j=1}^{N}\int_0^Tu_j(T,x)\psi_j^T(x)
	&=\sum_{j=1}^m\int_0^T\partial_{x}\varphi_{j}(t,l_j)\psi_{jx}(t,l_j)+\sum_{j=m+1}^N\int_0^T\partial_{x}^{2}\varphi_{j}(t,l_j)\psi_{jxx}(t,l_j).
\end{align*}
From the Propositions \ref{reg. Adj.} and \ref{phi_xx in L^2}, we have that $B$ is continuous, giving (i).

\vspace{0.2cm}
\noindent\textbf{(ii) If B is coercive, the exact controllability holds, that is, $u(T,\cdot)=u^T$.}
\vspace{0.2cm}

Indeed, assume for a moment that $B$ is coercive, i.e., the observability inequality \eqref{Obs. Ineq. 2} holds\footnote{For a detailed discussion on the equivalence between coercivity and the observability inequality, we recommend consulting \cite{Lions}.}. Then given $u^T\in \mathbb{L}^2(\mathcal{T})$, from the Lax-Milgram theorem, there exists $\varphi^T \in \mathbb{L}^2(\mathcal{T})$ such that
\begin{align*}	\sum_{j=1}^{N}\int_0^{l_j}u_j^T\psi_j^T=B(\varphi^T,\psi^T),\ \forall\psi^T\in \mathbb{L}^2(\mathcal{T}).
\end{align*}
Thus, for $\varphi(t,\cdot)=S(T-t)\varphi^T$, $p_j=-\partial_{x}^{2}\varphi_{j}(\cdot,l_j)$ and $g_j=\partial_{x}\varphi_{j}(\cdot,l_j)$ the solution $u$ of \eqref{LkdV-2} satisfies
\begin{align*}	\sum_{j=1}^{N}\int_0^{l_j}u_j^T\psi_j^T=B(\varphi^T,\psi^T)=\sum_{j=1}^m\int_0^Tg_j(t)\psi_{jx}(t,l_j)-\sum_{j=m+1}^N\int_0^Tp_j(t)\psi_{jxx}(t,l_j),
\end{align*}
for every $\psi^T\in \mathbb{L}^2(\mathcal{T})$, which implies, by Lemma \ref{contr. lemma}, that $u(T,\cdot)=u^T$, which ensures (ii).
\vspace{0.1cm}


\section{The set \texorpdfstring{$\mathcal{N}^{\dagger}$ is non empty and countable}{Countable critical}}\label{app: countable ciritcal}

In this part, we prove that the new set of critical length, defined in \eqref{critical-3}, is non-empty and countable. Recall that $L \in \mathcal{N}^\dagger$ if it can be written as $L^2=-\left(w_{1}^2+w_{1}w_{2}+w_{2}^2\right)$ where
\begin{align}
\label{A1}
(w_{1},w_{2})\in \mathbb{C}^2, \text{ \   such that \ } w_{1}^{2}e^{w_{1}}=w_{2}^{2}e^{w_{2}}=(w_{1}+w_{2})^{2}e^{-(w_{1}+w_{2})}.
\end{align}
In particular, we show a direct relation between this set and $\mathcal{N}^{*}$, defined in \eqref{critical-2}. Let $(w_{1},w_{2})$ satisfying \eqref{A1}, take $c=w_{i}^{2}e^{w_{i}}=(w_{1}+w_{2})^{2}e^{-(w_{1}+w_{2})}$, $i=1,2$ then we have that $z_{i}=\frac{w_{i}}{2}$ is solution of $z_{i}e^{z_{i}}=\pm (\frac{c}{4})^{1/2}$. Similarly, $(z_{1}+z_{2})^{2}e^{-2(z_{1}+z_{2})}=\frac{c}{4}$, which implies $(z_{1}+z_{2})e^{-(z_{1}+z_{2})}=\pm (\frac{c}{4})^{1/2}.$ Then, by symmetry, we have either one of the following cases 
\begin{itemize}
    \item $z_{1}e^{z_{1}}=z_{2}e^{z_{2}}=-(z_{1}+z_{2})e^{-(z_{1}+z_{2})}$.
    \item $z_{1}e^{z_{1}}=-z_{2}e^{z_{2}}=(z_{1}+z_{2})e^{-(z_{1}+z_{2})}$.
    \item $z_{1}e^{z_{1}}=z_{2}e^{z_{2}}=(z_{1}+z_{2})e^{-(z_{1}+z_{2})}$.
\end{itemize}
The first case is the equation related to the critical set $\mathcal{N}^{*}$; therefore, it has a countable number of solutions. For the second and third cases, we can follow \cite[Proposition 3 and 4]{GG1} to ensure that it has a countable number of solutions.

Finally, observe that if $L\in \mathcal{N}^{*}$, then $2L\in \mathcal{N}^{\dagger}$. If $L\in \mathcal{N}^{*}$, then $L^2=-\left(z_{1}^2+z_{1}z_{2}+z_{2}^2\right)$, for some$(w_{1},w_{2})\in \mathbb{C}^2$ such that
\[z_{1}e^{z_{1}}=z_{2}e^{z_{2}}=-(z_{1}+z_{2})e^{-(z_{1}+z_{2})}.\]
It is easy to see that $w_i=2z_i$, for $i=1,2$, satisfy \eqref{A1}. Moreover
\[-\left(w_{1}^2+w_{1}w_{2}+w_{2}^2\right)=-4\left(z_{1}^2+z_{1}z_{2}+z_{2}^2\right)=(2L)^{2}.\]

\subsection*{Acknowledgments} We would like to express our sincere gratitude to the reviewer for the positive report, and to the editor for handling our manuscript.  This work is part of the Ph.D. thesis of da Silva at the Department of Mathematics of the Universidade Federal de Pernambuco.





%

	\end{document}